\documentclass[11pt]{article}
\usepackage[utf8]{inputenc}
\usepackage{setspace}
\usepackage[margin=1.0in]{geometry}
\usepackage{amssymb, amsmath}
\usepackage{amsthm}
\usepackage{colordvi,verbatim,hyperref}
\usepackage{color,enumerate}
\usepackage{graphicx}
\usepackage{caption}
\usepackage{algorithm,algorithmic}
\usepackage{subcaption}
\usepackage[justification=centering]{caption}
\usepackage{authblk}
\usepackage{lineno}
\usepackage[toc,page]{appendix}

\theoremstyle{definition}
\newtheorem{defn}{Definition}

\newtheorem{lm}{Lemma}
\newtheorem{thrm}{Theorem}
\allowdisplaybreaks

\begin{document}
\title{\textbf{Interdicting Restructuring Networks with Applications in Illicit Trafficking}\footnotetext[0]{Email addresses: d.kosmas@northeastern.edu (Daniel Kosmas), tcshark@clemson.edu (Thomas C. Sharkey), mitchj@rpi.edu (John E. Mitchell), k.maass@northeastern.edu (Kayse Lee Maass), mart2114@umn.edu (Lauren Martin)}}
\author[a]{Daniel Kosmas\footnote{Corresponding author (d.kosmas@northeastern.edu).}}
\author[b]{Thomas C. Sharkey} 
\author[c]{John E. Mitchell}
\author[a]{Kayse Lee Maass}
\author[d]{Lauren Martin}
\affil[a]{\footnotesize Department of Mechanical and Industrial Engineering, Northeastern University, 360 Huntington Avenue, 334 SN, Boston, MA 02115, USA}
\affil[b]{\footnotesize Department of Industrial Engineering, Clemson University, Freeman Hall, Clemson University
Clemson, SC 29634-0920, USA}
\affil[c]{\footnotesize Department of Mathematics, Rensselaer Polytechnic Institute, 110 8th Street, Troy, NY 12180, USA}
\affil[d]{\footnotesize School of Nursing, University of Minnesota, 308 SE Harvard St, Minneapolis, MN 55455, USA}
\date{}
\maketitle
\vspace{-1.5cm}
\begin{abstract}
    We consider a new class of max flow network interdiction problems, where the defender is able to introduce new arcs to the network after the attacker has made their interdiction decisions. We prove properties of when this restructuring will not increase the value of the minimum cut, which has important practical interpretations for problems of disrupting drug trafficking networks. In particular, it demonstrates that disrupting lower levels of these networks will not impact their operations when replacing the disrupted participants is easy.  For the bilevel mixed integer linear programming formulation of this problem, we devise a column-and-constraint generation (C\&CG) algorithm to solve it.  Our approach uses \emph{partial information} on the feasibility of restructuring plans and is shown to be orders of magnitude faster than previous C\&CG methods. We demonstrate that applying decisions from standard max flow network interdiction problems can result in significantly higher flows than interdictions that account for the restructuring.\newline \footnotesize Keywords: OR in societal problem analysis; network interdiction; drug trafficking
\end{abstract}
\section{Introduction}
Network interdiction is a two player Stackelberg game \cite{stackelberg} between an interdictor (attacker) and an operator (defender). For the maximum flow network interdiction problem (MFNIP), the defender wishes to maximize the flow through the network, while the attacker attempts to minimize the flow by disrupting nodes and arcs (note that, without loss of generality, we can focus on the disruption of nodes in the network), removing them from the network \cite{orig}. Max flow network interdiction problems model situations where we want to move goods or services through a network, such as telecommunications, electrical power, transportation, and illicit trafficking networks \cite{survey}. Illicit trafficking networks, such as drug trafficking and human trafficking networks, are networks where the product or service being supplied is illegal. In illicit trafficking networks, the interdiction decisions would represent disruptions to the network, with the simplest example being law enforcement arresting participants within the network. It is important to recognize that not all participants in an illicit trafficking network, such as a drug trafficking network, are engaged in criminal activity, and that non-law enforcement based interventions, such as investment in communities and rehabilitation, can also be modeled as interdiction decisions.  Therefore, for the purposes of this paper, the term `interdiction' will be defined to be any disruptive action that alters the nodes and arcs within the illicit trafficking network. 

An important class of network interdiction models, known as defender-attacker-defender models, allows for the defender to protect part of their network before the attacker makes interdiction decisions \cite{DAD, lozano2017backward}. However, these models do not allow for the network to change after the interdiction decisions have been made. Extending these models over a longer time period, once the network has been disrupted, the damage is permanent, and no attempts to repair the network are allowed \cite{multi}. In some applications, such as with illicit trafficking networks, this would contradict our intuition; if an illicit trafficking network has lost participants, and thus lost business, we would expect the network to restructure to compensate for those losses. Incorporating the restructuring decisions of the illicit trafficking network is important to providing appropriate recommendations on how law enforcement and other stakeholders should disrupt the illicit trafficking network.  However, fully incorporating the dynamic back-and-forth between attackers and defenders is computationally difficult \cite{dynamicMFNIP,dynamicSPI}. In this paper, we offer an approach to capture more of the dynamics than traditional MFNIPs which helps to yield important practical insights into the problems of disrupting drug trafficking networks. We consider a variant of MFNIP where the defender is allowed to restructure parts of the network before operating it.

The dynamics of the network are important to consider in disrupting illicit trafficking networks \cite{bright2019illicit}. While MFNIP can be applied to any illicit trafficking network, here we examine the MFNIP in relation to drug trafficking networks. Drug trafficking results in billions of dollars in illicit profits, constituting a fifth of all revenue from organized crime \cite{chandra2020illicit}. It has been observed that when law enforcement targets low-level drug offenders, such as dealers or users, there is little impact on the amount of drugs sold. This is because new participants can be easily recruited to replace those who were incarcerated \cite{prison}. Current network interdiction models for disrupting drug trafficking networks do not account for the ability to replace nodes and arcs. Similarly, Magliocca et al.\ \cite{magliocca2019modeling} developed an agent-based model to predict the response of drug traffickers to interdiction efforts by an individual agent. Our model uses the full operations of the drug trafficking network to plan global disruptions.

We provide a new class of network interdiction problems where the defender is allowed to restructure arcs (without loss of generality, this also models adding new nodes to the network since a node can be modeled as a pair of nodes and an arc) in the network after interdiction decisions have been made. We first describe some properties on how restructuring impacts the optimal flow. These observations yield important insights about how disrupting lower levels of drug trafficking may not impact their operations when replacing the disrupted participants is easy for the network.  We then formulate a bilevel integer program (BIP) to model the problem. We provide a column-and-constraint generation (C\&CG) algorithm to solve this class of problems, and compare the interdiction decisions to those of the standard max flow network interdiction problem. By utilizing partial information from previously explored restructuring decisions, we are able to solve this problem faster than previous C\&CG algorithms.

\subsection{Literature Review}

Network interdiction has been previously applied to the disruption of drug trafficking networks. Malaviya et al.\ \cite{multi} first applied network interdiction models to the disruption of city-level drug trafficking networks and described a framework for generating networks representative of city-level drug trafficking networks. Baycik et al.\ \cite{info} expanded upon this work to include the disruption of drug trafficking networks where not all participants in the network directly handle the drugs. In this work, an additional information network models how bosses direct the flow of drugs in the drug trafficking network. Shen et al.\ \cite{launder} also expands upon this work by also including a `dirty money' network and a money laundering network to help understand the role of money laundering in drug trafficking networks. In these works, the drug trafficking network does not have the ability to react to the disruptions. However, we can reasonably expect that the members of the drug trafficking network will have some sort of reaction, from drug dealers finding new users to sell their drugs to, to the operations of the suppliers being moved to backup locations. To be able to properly disrupt drug trafficking networks, we must make disruption decisions that incorporate likely reactions and actions the networks may make.  We will demonstrate in our computational analysis the importance of accounting for the reaction of the drug trafficking network.  

Uncertainty in max flow network interdiction has been a topic of interest for many years \cite{cormican, janjarassuk}. More recently, Lei et al.\ \cite{stochastic} explored a variant of stochastic MFNIP where the defender is able to add capacity to arcs in the network after the uncertainty around capacities is resolved. They allow for continuous capacity to be added, which does not introduce the challenges of integer decisions in the lower level problem. Borrero et al.\ \cite{borrero2016sequential,borrero2019sequential} explore uncertainty in network structure. Their model does not allow for the underlying network to change, which is the adaptation we introduce and critical for our motivating applications.

To our knowledge, this work is the first to explore the defender adding arcs to the network after interdictions have been made in max flow network interdiction. Our problem is similar to shortest path interdiction with improvement (SPIP-I) introduced by Holzmann and Smith \cite{spipi}, where arcs can be improved after interdiction decisions have been made. In their work, they construct a layered network to solve their problem as a shortest path interdiction problem. Fischetti et al.\ \cite{monotone} further explores interdiction games, where the decisions of the interdictor explicitly restrict the actions of the operator under the assumption of downward monotonicity. In this work, they are able to apply a direct penalization to the objective function value of previous sub-problem solutions based on the current interdiction plan due to the interdiction decisions directly preventing the defender from taking certain actions. In our problem, the interdiction decisions do not restrict the actions of the defender, but instead allow for certain arcs to be added to the network. Further, generating constraints associated only with the objective function value based on each restructuring would not be correct, as the minimum cut may not be the same across different restructuring plans. This prevents their method from being applied to our problem. The interdiction problems studied in Contardo and Sefair \cite{contardo2019progressive} similarly have that the interdiction decisions directly penalize the actions of the defender, and likewise their solution method cannot be applied to our problem. 

Our proposed network interdiction problems fall into the general class of bilevel integer programming problems. In general, bilevel integer programs are $\mathcal{NP}$-Hard problems, and are more difficult to solve than standard optimization problems because of the interactions between the two decision makers. When decisions are restricted to be integers, these problems become harder, because standard methods to solve non-integer counterparts cannot be used. Solution methods to general BIPs have been recently explored \cite{bilevelF, lozano2017value, bilevelT}. BIPs have been solved via C\&CG methods, which separate the inner optimization into two problems, a `middle' problem with just integer decision variables and one with continuous variables. By enumerating over the feasible integer decisions, we can condense the remaining two optimization problems into one by exploiting duality of linear programs. We then solve the optimization problem with a subset of these feasible integer decisions, and iteratively identify new integer decisions to add to the problem. Smith and Song \cite{survey} identify C\&CG as a new method to solve network interdiction problems.

Zeng and An \cite{cncg1} first applied these techniques to problems with the relatively complete response property, where the defender’s problem could always be solved, regardless of the decisions the attacker makes. Intuitively, these methods allow for the attacker to account for how the defender has previously reacted to their decisions, resulting in better informed decisions by the attacker. These methods are improved upon in Yue et al.\ \cite{cncg2}, applying these methods to problems that do not satisfy this property. The intuitive extension with these methods is that only the previous defender decisions that are relevant to the current attacker decision are accounted for. While this starts to exploit information from previous defender decisions, we can gain more information. When we can utilize partial information from the defender decisions that are infeasible with respect to the current interdiction decision, we can gain unique information that may not be accounted for in the previous feasible defender decisions. We extend previous C\&CG methods to utilize partial information from previously identified defender decisions and, in particular, when a defender decision is not fully feasible with respect to the new attacker decision. Utilization of partial information from previous defender decisions can also be implemented in general solvers. The method of Lozano and Smith \cite{lozano2017value} similarly seeks to block previously identified defender decisions, where utilizing partial information would instead only block the infeasible components of previously identified defender decisions.

\subsection{Paper Organization}
The paper is organized as follows: Section \ref{statement} formally introduces the max flow network interdiction problem with restructuring, describes properties on when restructuring will not increase the flow, and interprets these properties in the context of disrupting drug trafficking networks. Section \ref{modeling} discusses how drug trafficking networks can be modeled as max flow network interdiction problems, and how we model when the network is able to restructure. Section \ref{algorithm} derives a column-and-constraint generation algorithm that utilizes partial information from previous iterations. Section \ref{compare} compares this algorithm to other column-and-constraint generation algorithms, and Section \ref{analysis} analyzes the results on drug trafficking networks. Section \ref{conclude} concludes the paper and discusses directions of future work.

\section{Problem Statement and Theoretical Policy Analysis}
\label{statement}
\subsection{Problem Description and a Bilevel Integer Programming Formulation}
The max flow network interdiction problem with restructuring (MFNIP-R) is played over three sets of decisions, one for the attacker then two for the defender. This game is played on a network $G = (N, A)$, where $N$ is the set of nodes, $A$ is the set of arcs initially in the network, and $A^R$ is the set of restructurable arcs, which are arcs that can be added to the network. Nodes and arcs are assigned capacities $u: N \cup A \cup A^R \rightarrow \mathbb{R}_+$. In this work, we consider $A \cap A^R = \emptyset$. Flow is sent along the arcs from a source node $s$ to a sink node $t$. Without loss of generality, we assume there are no arcs into $s$ and out of $t$. First, the attacker chooses a set of nodes to disrupt in the network. Disrupting a node is a binary decision that sets the capacity of that node to $0$. Second, the defender chooses a set of arcs from $A^R$ to restructure, adding those arcs in the network. Restructuring an arc is a binary decision that changes that arc from having $0$ capacity to its specified non-zero capacity. Lastly, the defender maximizes the flow through the restructured network. While we could consider the last two sets of decisions as a single set of decisions since they belong to the same decision-maker working towards the same goal in sequence, by considering them separately, we can develop a decomposition algorithm that leverages strong duality of linear programs. For a background in modeling max flow problems as linear programs, including converting problems with multiple source/sink nodes to single source/sink nodes, see \cite{introduction}.

Figure \ref{fig:sample} displays a sample hierarchical network, which is how city-level drug trafficking networks have previously been modeled. Figure \ref{fig:sampleNet} is the original network, and Figure \ref{fig:sampleRest} is the set of restructurable arcs that may be added to the network. In the context of drug trafficking, the triangle node would be the supplier, the square nodes would be the drug dealers, and the circle nodes would be the drug users. In a max flow model, the supplier would be the source node, and the drug users would be sink nodes.
\begin{figure}[h!]
\centering
\begin{subfigure}{.45\textwidth}
  \centering
  \includegraphics[width=\linewidth]{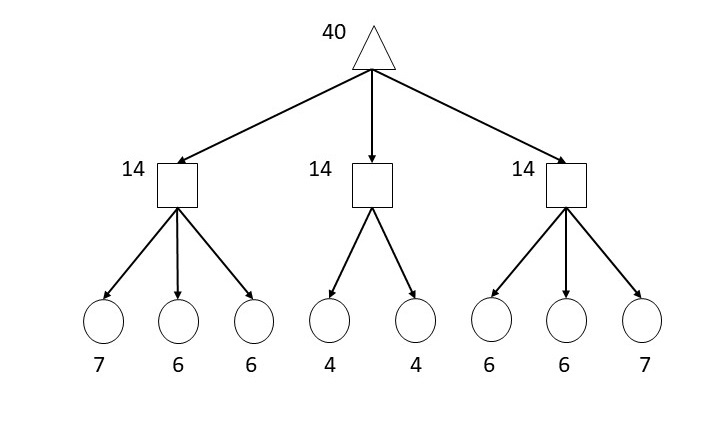}
  \caption{Sample network with node capacities}
  \label{fig:sampleNet}
\end{subfigure}%
\begin{subfigure}{.45\textwidth}
  \centering
  \includegraphics[width=\linewidth]{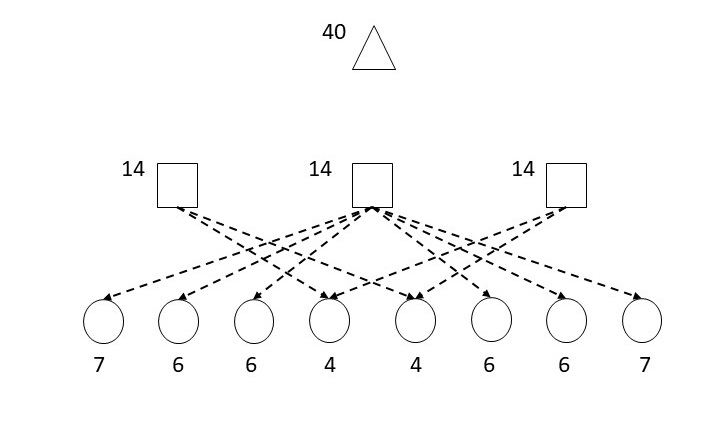}
  \caption{Set of restructurable arcs for sample network}
  \label{fig:sampleRest}
\end{subfigure}
\caption{Example network motivated by city-level drug trafficking networks}
\label{fig:sample}
\end{figure}

We note that, traditionally, max flow network interdiction models disrupt arcs instead of nodes. However, in previous work on disrupting city-level drug trafficking networks \cite{info, multi, launder}, their models disrupt nodes, since the interdiction decisions are modeling law enforcement disrupting the operations of individuals involved in the network, not communications or the passing of drugs between the individuals. Malaviya et al.\ \cite{multi} demonstrates the equivalency between interdicting nodes and interdicting arcs.  Our results can be extended to more general problems that allow for both node and arc interdiction; however, we focus on node interdictions given our motivating application. 

We can model MFNIP-R as a bilevel integer program with integer decisions in both the upper and lower level problems. The upper level (attacker's) problem is the problem of finding the interdiction plan that minimizes the restructured flow, and the lower level (defender's) problem is the problem of finding the restructuring plan and maximum flow in the restructured network. Let $x_{ij}$ be the variable representing the amount of flow on arc $(i,j) \in A$ and $x_i$ be the amount of flow through node $i \in N$, $y_{i}$ be the indicator of whether or not node $i \in N$ has been interdicted, and $z_{ij}$ be the indicator of whether or not arc $(i,j) \in A^R$ has been restructured. More formally,\begin{singlespace}
\begin{equation*}
    y_{i} = \begin{cases} 1 \text{ if node } i \text{ has been interdicted,}\\
    0 \text{ otherwise, }
    \end{cases}
\end{equation*}
\end{singlespace}
and\begin{singlespace}
\begin{equation*}
    z_{ij} = \begin{cases} 1 \text{ if arc } (i,j) \text{ has been restructured,}\\
    0 \text{ otherwise.}
     \end{cases}
\end{equation*}
\end{singlespace}
Let $Y \subseteq \{0, 1\}^{|N|}$ be the set of all feasible interdiction decisions, and $Z(y) \subseteq \{0, 1\}^{|A^R|}$ be the set of all feasible restructuring decisions with respect to interdiction decision $y \in Y$, and assume these sets are finite. The following BIP models the attacker's problem:
\vspace{-8mm}
\begin{singlespace}
\begin{align}
    \label{model1}
    \min_{y \in Y} \max_{x, z} ~~~ & \sum_{(s,i) \in A} x_{si} \nonumber\\  
    \text{s.t.} \hspace{.050cm }& \sum_{(j, i) \in A \cup A^R} x_{ji} -  x_{i} = 0 & \text{ for } i \in N\setminus \{s,t\}\nonumber\\
    & x_i - \sum_{(i, j) \in A \cup A^R} x_{ij} = 0 & \text{ for } i \in N\setminus \{s,t\} \nonumber\\
    & 0 \leq x_{ij} \le u_{ij} & \text{ for } (i,j) \in A\\
    & 0 \leq x_{i} \le u_i (1 - y_{i}) & \text{ for } i \in N\setminus \{s,t\} \nonumber\\
    & 0 \leq x_{ij} \le u_{ij}z_{ij} & \text{ for } (i,j) \in A^R \nonumber\\
    & z \in Z(y) \nonumber 
\end{align}
\end{singlespace}
The objective function is the total amount of flow out of the source node. The first two sets of constraints are flow balance constraints, ensuring the flow into a given node is equivalent to the flow out of that node. The next two sets of constraints are capacity constraints, ensuring that the amount of flow traveling along an arc or node is at most the capacity of that arc or node. We do not define these constraints for the source and sink nodes, as the remaining flow balance constraints will ensure that the flow out of the source will be the same as into the sink, and the source and sink are uncapacititated. The max flow problem can be equivalently modeled as adding an arc from the sink to the source and finding the maximum circulation through the network, in which case these additional flow balance constraints would be necessary. We note that if $z_{ij} = 0$ for all $(i,j) \in A^R$, then the problem reduces to the standard MFNIP. In Section \ref{modeling}, we discuss how to model city-level drug trafficking networks, both in defining the set of arcs and set of restructurable arcs for these networks, and in defining constraints on interdictions and restructurings. Before we do so, we discuss properties of the problem for general networks.

\subsection{Properties of Problem}
We consider how the interdiction decisions can trigger restructuring decisions and how the overall effects of these decisions can impact the max flow in the network. We will show that only certain classes of restructuring decisions can be guaranteed to \emph{not increase} the maximum flow.  This will help shed insight into `bad' disruption strategies for our motivating applications in drug trafficking. We consider the example network in Figure \ref{fig:sample} that helps to illustrate that implementing available interdiction decisions (even when our budget allows it) may not be optimal and, counter-intuitively, increase the flow in the network. This example is based on modeling city-level drug trafficking. For this example, nodes are labeled with $(flow, capacity)$, and current arcs are uncapacitated. Figure \ref{fig:sfig1} demonstrates the maximum flow through this network.

In the example network, we only have the budget to interdict one node in the bottom tier. The set of restructurable arcs includes arcs that start at the left or right node in the middle tier and end at the capacity 4 nodes in the bottom tier, and arcs that start at the middle node in the middle tier and end at any node it does not currently reach. The defender has the budget to restructure any one arc. We assume that only arcs where one endpoint has an interdicted neighbor can be restructured, as we specifically want to model how the impacted individuals react to the disruptions; in modeling a drug trafficking network spanning multiple regions, the removal of a drug dealer in one region would not impact the operations of a dealer in a different region.

For the MFNIP in this network, as demonstrated in Figure \ref{fig:sfig2}, it is optimal to interdict one of the two nodes with a capacity of $4$. However, once we incorporate restructuring, interdicting that node allows for the dashed arc in \ref{fig:sfig3} to be restructured, allowing for flow to be shifted in a way that increases the total flow through the network. With this set of restructurable arcs, it is optimal to interdict either node with capacity $7$, as the arcs that would be allowed to restructure after this interdiction end at nodes already at capacity. Further, note that if we could \emph{only} interdict the nodes with capacity $4$, instead of being able to interdict a node in the bottom tier, then it is optimal to not implement any interdiction decisions, as the resulting restructuring will always increase the flow to be more than in the original network.

\begin{figure}[h!]
\centering
\begin{subfigure}{.45\textwidth}
  \centering
  \includegraphics[width=\linewidth]{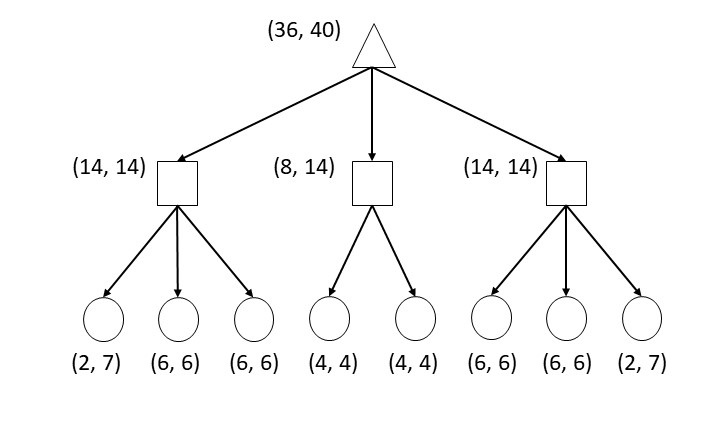}
  \caption{Original optimal flow}
  \label{fig:sfig1}
\end{subfigure}%
\begin{subfigure}{.45\textwidth}
  \centering
  \includegraphics[width=\linewidth]{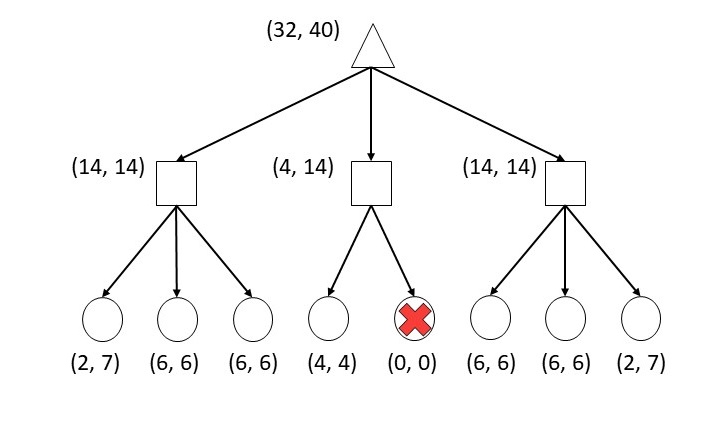}
  \caption{Optimal interdiction decision under MFNIP}
  \label{fig:sfig2}
\end{subfigure}

\begin{subfigure}{.45\textwidth}
  \centering
  \includegraphics[width=\linewidth]{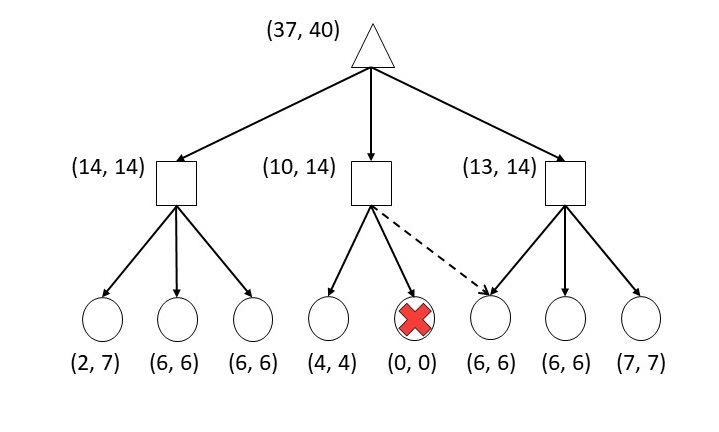}
  \caption{Restructuring decision after optimal MFNIP interdictions}
  \label{fig:sfig3}
\end{subfigure}%
\begin{subfigure}{.45\textwidth}
  \centering
  \includegraphics[width=\linewidth]{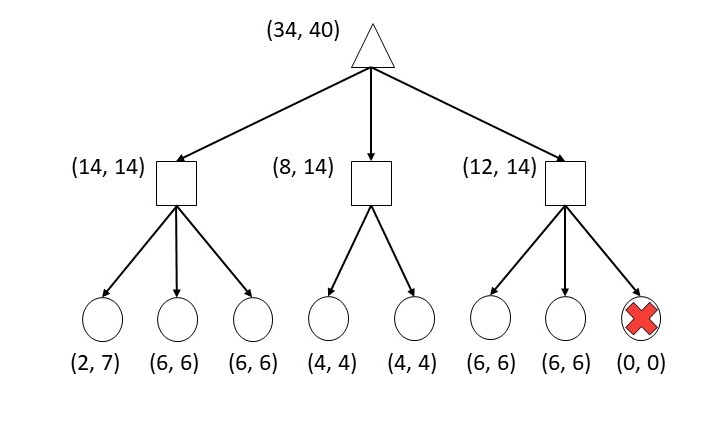}
  \caption{Optimal interdiction decision accounting for restructurings}
  \label{fig:sfig4}
\end{subfigure}
\caption{How the optimal interdiction plan changes when restructuring is allowed in an example network}
\label{fig:example}
\end{figure}

In this example network, there are some nodes in the bottom tier that have a flow less than their capacity. Our model would not allow arcs to be restructured ending at these nodes, even though one might suspect that, in practice, the nodes would want to. We can think of this as modeling a closed market, where dealers may not want to sell to users they are unfamiliar with, unless the users they currently sell to have been interdicted \cite{potter2009exploring}.

We first wish to find cases for interdiction decisions under which the value of the flow will not increase, regardless of restructuring decisions. We do so by considering the auxiliary network $D_f = (N, A_f)$ of the optimal maximum flow in the current network. We construct this auxiliary network by the following procedure for every $(i,j) \in A$. If $x_{ij} < u_{ij}$, then $(i,j) \in A_f$. If $x_{ij} > 0$, then $(j, i) \in A_f$. Define the corresponding minimum $(s,t)$-cut to be $[U, \bar{U}]$, where a node $i \in U$ if $i$ can be reached from $s$ in $D_f$, and $i \in \bar{U}$ otherwise. We consider the type of an arc $(i,j)$ that is allowed to be restructured. The proof of the lemma appears in Appendix \ref{apdx:pf}.

\begin{lm}
\label{lm:noInc}
Consider an arc $(i,j) \in A^R$. If any of the following conditions hold, restructuring $(i,j)$ will not increase the flow:
\begin{enumerate}
    \item If $i \in \bar{U}$,
    \item If $i \in U$ and $j \in U$, 
    \item If $i \in U$ and $j \in \bar{U}$ and there is no path from $j$ to $t$ in $D_f$.
\end{enumerate}
Additionally, if restructuring $(i,j)$ will not increase the flow, then $(i,j)$ must satisfy one of these three conditions.
\end{lm}

These results can be combined into more powerful results about sets of nodes allowed to restructure based on the selected interdiction decisions. 

\begin{thrm}
\label{thrm:no_inc}
Let $I$ be the set of arcs that are allowed to restructure based on the selected interdiction decisions, where all arcs in $I$ satisfy one of Conditions $1$, $2$, or $3$ of Lemma \ref{lm:noInc}. 
\begin{enumerate}
    \item If all $(i,j) \in I$ satisfy Condition $1$ or Condition $2$ of Lemma \ref{lm:noInc}, then the value of the flow will not increase.
    \item If there is an arc $(i,j) \in I$ satisfying Condition $3$ of Lemma \ref{lm:noInc} and there is no subset of arcs $J \subseteq I$ that would complete a path from $s$ to $t$ in $D_f$ that includes $(i,j)$, then the value of the flow will not increase.
\end{enumerate}
\end{thrm}

We can interpret these results in terms of \emph{restructuring across the minimum cut} for the original maximum flow.  In particular, Conditions $1$ and $2$ of Lemma \ref{lm:noInc} ensure that the restructuring decisions do not introduce arcs that cross the minimum cut. Condition $3$ ensures that if the restructured arc crosses the minimum cut, it does not enter a node with a path to the sink, where this path may include additional restructurable arcs $J$. We note that, the remaining case, when $i \in U$, $j \in \bar{U}$, and there is a path from $j$ to $t$ in $D_f$ will result in an increase in flow, since adding arc $(i,j)$ will complete a path from $s$ to $t$ in $D_f$. However, the restructuring of such an arc (or set of arcs) may not necessarily result in an increase in flow greater than the decrease in flow caused by interdictions. For example, if the circle nodes in Figure \ref{fig:example} with capacity $7$ instead had capacity $5$, the amount of flow recovered from restructuring would be $3$, resulting in a net decrease in flow.

This interpretation has important policy implications for disrupting drug trafficking networks.  For city-level drug trafficking networks, Malaviya et al.\ \cite{multi} have shown that the minimum cut in the network tends to be be between the drug dealer and user level.  Therefore, our results suggest that for this application, a focus on interdictions at this level will not impact the overall prevalence of drug trafficking.  This mathematically confirms the analysis of criminal justice scholars \cite{prison}.

\section{Modeling Restructuring in Drug Trafficking Networks}
\label{modeling}
We focus our discussion on modelling hierarchical networks, since the field has found value in viewing drug trafficking networks as hierarchical networks. Each tier is representative of the node's role in the functionality of the network \cite{info, multi}. The first (bottom) tier is representative of drug users, the second tier is representative of drug dealers, the third tier is representative of safe houses of distinct organizations, while the tiers above that are representative of management in distinct organizations \cite{multi}. When we refer to a \emph{safe house}, we are referring to the individual(s) responsible for operating it and ensuring its role as a distributor of drugs to dealers. Distinct organizations are constituted from nodes in the safe house tier through the top tier, and there are no arcs that start at a node in one organization and end at a node in a different organization. The nodes in the highest tier for each organization are the suppliers for the organization. Arcs in these networks always start at the $i^{th}$ tier and end at the $(i - 1)^{th}$.

\subsection{Modeling Interdiction Decisions}
We first define the set of feasible interdiction decisions $Y$. The first constraint is one that all interdiction models include, a budget constraint. Suppose $c_{i}$ is the cost to interdict node $i$, and a budget $b$. We define the budget constraint as:
\begin{equation}
    \sum_{i \in N} c_{i} y_{i} \le b\text{.}
\end{equation}

In disrupting drug trafficking networks, we need to interdict participants in a way that allows us to gather information about the participants above them. Malaviya et al.\ \cite{multi} model this as climbing the ladder constraints. Let $\tau_i$ be the number of child nodes of node $i$ that need to be interdicted before node $i$ can be interdicted. Climbing the ladder constraints can be defined as:
\begin{equation}
    \tau_i y_{i} \le \sum_{(i,j)\in A} y_{j} \text{ for all } i \in N\text{.}
\end{equation}

It may also be that we want to enforce that some nodes from a set of nodes are required to be interdicted. In the application of drug trafficking, we may want to understand how disrupting the high-level operations of the network triggers restructuring decisions. Let $T$ be the set of nodes we want to interdict at least $g$ nodes from. We can define this constraint as:
\begin{equation}
    \sum_{i \in T} y_{i} \ge g\text{.}
\end{equation}
In our application, we want to understand how disrupting the leadership of larger organizations impacts the restructuring of the network. In order to do this, we define $T$ to be the set of nodes in the top two tiers of the two largest organizations. We refer to this as `interdicting leadership'.

\subsection{Modeling Restructuring}
We now focus on how we model the reactions of the network to interdiction efforts. We first define what arcs belong to the set $A^R$. Since city-level drug trafficking are modeled as hierarchical networks, we define the base set of restructurable arcs to follow the same rules as arcs currently in the network do. This means that the initial set $A^R$ consists of arcs that could have belonged in the original arc set. These arcs would be representative of participants in the drug trafficking network forming connections with new individuals to sell their product to.

Before defining the constraints on the defender's decisions, we wish to capture which participant is initiating the restructuring efforts, not just the appearance of new arcs. We do so by defining $z = (z^{in}_{ij}, z^{out}_{ij})$, based on whether a node is able to restructure to a tier above or below them, respectively. More formally, $z^{in}_{ij}$ is the indicator of restructuring arc $(i,j)$ when node $j$ is able to restructure, and $z^{out}_{ij}$ is the indicator of restructuring an arc when node $i$ is able to restructure. For applications in disrupting drug trafficking networks, we model that participants in the network may seek to replace connections when either those above them (e.g., a user replacing a dealer) or below them (e.g., a safe house/distributor replacing a dealer) have been interdicted. Therefore, it is important to note whether the higher level or lower level participant is reaching out to establish the new connection, as we will assume there is a limit to the number of re-connections a participant can establish. 

We now describe constraints on the decisions of the defender, $Z(y)$. We assume that they have structured their current network (pre-interdiction) systematically and thus will only allow a small number of replacements to occur if the interdictions impact the participants a node is directly connected to. First, we want to restrict the number of arcs leaving node $i$. In the drug trafficking network, this indicates that each participant's resources to restructure their network are limited or they desire to only react in a small way to limit the number of new participants that know about them. We only allow $l_i$ new arcs to be created based on actions taken by given node $i$,  signifying that each participant only has time to develop a limited number of new connections below them. We can capture this with the following constraint:
\vspace{-5mm}
\begin{singlespace}
\begin{equation}
    \label{leave}
    \sum_{(i,j) \in A^R} z^{out}_{ij} \le l_i \text{ for all } i \in N\text{.}
\end{equation}
\end{singlespace}

Next, we want to restrict the number of arcs entering node $j$. In our drug trafficking network example, this could be represented multiple ways. At the lowest level, a drug user does not need a large number of dealers, and may want to limit the number of dealers they purchase from in order to limit the number of risky situations they put themselves in. Similarly for higher levels, participants may want to work with a smaller number of people so that the they can build trust amongst one another and decrease the chances of a connection providing information about them to law enforcement. We only allow $s_j$ new arcs to be created by the actions of a given node $j$, limiting the number of connections a participant can form with participants above them. We can define the constraint limiting the number of arcs entering a node during the restructuring as:
\vspace{-3mm}
\begin{singlespace}
\begin{equation}
    \label{enter}
    \sum_{(i,j)\in A^R} z^{in}_{ij} \le s_j \text{ for all } j \in N\text{.}
\end{equation}
\end{singlespace}
The next restriction we make is to only allow nodes that are incident to an interdicted arc to replace a disrupted connection. This is to capture that it is likely that local `reactions' will occur within the network, i.e., if the interdiction did not disrupt the status quo for a certain segment of the network, this segment would have no reason to restructure itself.  For example, in order to make up for their lost earnings from losing a drug user, a drug dealer may seek out new users to sell to. A parameter $q_i$ for node $i$ determines how many new arcs are allowed to be added in and out of node $i$ \textit{per interdicted participant} connected to node $i$. Intuitively, interdicting a parent node allows for restructuring to a new parent node, and interdicting a child node allows for restructuring to a new child node.
We can model this constraint as:
\vspace{-3mm}
\begin{singlespace}
\begin{equation}
    \label{neighborModAbove}
    \sum_{(i,j)\in A^R} z^{out}_{ij} \le q_i \sum_{h:(i,h) \in A}  y_{h} \text{ for all } i \in N\text{.}
\end{equation}
\begin{equation}
    \label{neighborModBelow}
    \sum_{(i,j)\in A^R} z^{in}_{ij} \le q_j\sum_{h:(h,j) \in A}  y_{h} \text{ for all } j \in N\text{.}
\end{equation}
\end{singlespace}
Observe that if no child node of a node $i$ has been interdicted, then $y_{h}=0$ for all $h$, $(i,h)\in A$. Thus, the sum on the right side of the inequality \eqref{neighborModAbove} will be zero. This then forces $z^{out}_{ij}=0$ for all $j$, $(i,j) \in A^R$, capturing the desired behavior. The same behavior occurs for parent nodes of a node $j$. 

In order to not have the same arc added twice, we add the following constraint to ensure that both $z_{ij}^{out} = 1$ and $z_{ij}^{in} = 1$ cannot happen.
\vspace{-3mm}
\begin{singlespace}
\begin{equation}
    \label{noDouble}
    z^{in}_{ij} + z^{out}_{ij} \le 1 \text{ for all } (i, j) \in A^R\text{.}
\end{equation}
\end{singlespace}

The last constraint we have is a budget constraint. Similarly to the attacker having a budget $b$ to interdict the network with, we limit the amount of resources to restructure the entire network, where $a_{ij}$ is the amount of resources needed to restructure arc $(i,j)$, and $r$ is the maximum amount of resources available to restructure the network. This gives us the final constraint:
\vspace{-3mm}
\begin{singlespace}
\begin{equation}
    \label{budget}
    \sum_{(i,j)\in A^R} a_{ij} z_{ij} \le r\text{.}
\end{equation}
\end{singlespace}

With these four sets of constraints, we are able to capture the behavior of the defender reacting to attacker's decisions. We formalize the set $Z(y)$ as:
\vspace{-5mm}
\begin{singlespace}
\begin{align*}
    &\sum_{(i,j) \in A^R} z^{out}_{ij} \le l_i &\text{ for all } i \in N \\
    &\sum_{(i,j)\in A^R} z^{in}_{ij} \le s_j &\text{ for all } j \in N \\
    &\sum_{(i,j)\in A^R} z^{out}_{ij} \le q_i\sum_{h:(i,h) \in A} y_{h} &\text{ for all } i \in N \\
    &\sum_{(i,j)\in A^R} z^{in}_{ij} \le q_j\sum_{h:(h,j) \in A} y_{h} & \text{ for all } j \in N \\
    &z^{in}_{ij}+z^{out}_{ij} \le 1 &\text{ for all } (i,j) \in A^R \\
    &\sum_{(i,j) \in A^R} a_{ij} (z^{in}_{ij} + z^{out}_{ij}) \le r \\
    &z^{in}_{ij}, z^{out}_{ij} \in \{0, 1\} &\text{ for all } (i,j) \in A^R
\end{align*}
\end{singlespace}

\subsubsection{Adaptations on the Set of Restructurable Arcs}
The above defines our base instances for restructuring in city-level drug trafficking networks. We additionally consider two variants on top of this base set of arcs. We first discuss the recruitment of new participants, then discuss promotion of participants and recruitment across organizations (the two of which we refer to as organizational restructuring). We model the recruitment of new participants as nodes that are not yet connected to the source. For a new participant node $j$, we add the arc $(j, k)$ to the currently existing arc set, so the new participant can reach the layer below it. However, we add arcs $(i, j)$ to the set of restructurable arcs. For example, recruitable drug users would have an arc to the sink in the original network, and arcs to drug dealers in the restructurable set. Since there are no arcs $(i, j)$ in the network, there is no path from the source node to $j$, and thus no flow into $j$. To restructure arc $(i,j)$, we would need to interdict a node $h$ with $(i,h) \in A$, which, intuitively, is interdicting a currently existing child node of $i$ and, therefore, $h$ and $j$ would have similar roles in the network. When an arc $(i, j)$ is added to the network, there is the possibility of completing a path from the source, allowing the possibility of flow through $j$. This assumes each recruitable participant knows of other participants in the network, such as recruitable dealers knowing of users to sell to. Note that no new constraint needs to be added to include this modeling adaptation, as recruiting a new participant is incorporated in \eqref{neighborModAbove}.

We now seek to model organizational restructuring, which allows for the promotion of participants currently in an organization, as well as the recruitment of participants from one organization into another. We model promotion as including some arcs $(i,j) \in A^R$, where node $i$ is two tiers above $j$, instead of the typical one tier difference. When including this arc, we want to ensure that our promoted node is replacing one of its parent nodes. This means we only allow $z^{in}$ to take a non-zero value, as $z^{in}$ is allowed to be non-zero when a parent node is interdicted. Let $P$ be the set of promotable nodes. For all $p \in P$, we enforce 
\begin{equation}
    z^{out}_{ip} = 0 \text{ for all } (i, p) \in A^R \text{.}
\end{equation}

We model recruitment from one organization to another as including some arcs $(i,j) \in A^R$ where node $i$ belongs to one organization and node $j$ belongs to a different organization. As per the arcs described in the base instance, node $i$ is only one tier above node $j$. We want a node $j$ to be able to recruited only when the node they work for has been interdicted; a higher level participant has no incentive to join a different organization unless it is necessary for them to continue their operations. This again translates into only wanting $z^{in}_{ij}$ to be able to be non-zero. Let $C$ be the set of nodes eligible for cross-organizational recruitment. For all $c \in C$ we enforce
\begin{equation}
    z^{out}_{ic} = 0 \text{ for all } (i, c) \in A^R\text{.}
\end{equation}

We note that, only budget constraints are needed to define $Y$ and $Z(y)$. If $Z(y) = Z$ for all $y \in Y$, meaning our restructuring decisions are independent of interdiction decisions, our method reduces to the same method as \cite{cncg1}, which the method of \cite{cncg2} will also reduce to.

\section{Our Algorithm}
\label{algorithm}
We now derive our column-and-constraint (C\&CG) algorithm to solve MFNIP-R. In our C\&CG algorithm, we iteratively generate interdiction plans and minimum cut values associated with feasible restructuring plans constructed from previously considered restructuring plans, then determine what the optimal restructuring plan is to that interdiction plan. We note that, while our method is derived specifically for the application of disrupting city-level drug trafficking networks, the following derivation can be used for a general network.

In Yue et al.\ \cite{cncg2}, when a previously visited lower level integral solution is infeasible with respect to the current upper level solutions, we exclude the constraints in the upper level problem associated with the lower level solution. We propose a novel projection scheme based on the monotonicity of $Z(y)$ to find a non-trivial lower level decision that satisfies the constraints enforced by the new upper level decision.  In other words, we use \emph{partial information} obtained in previous iterations to tighten the bounds on the true objective function of a selected upper level decision. The method we derive is independent of network structure. As such, we will derive our method for a general network. Before proceeding with the description of the algorithm, we first define what is means for $Z(y)$ to be monotonic, as defined in Assumption 1 in \cite{monotone}.
\begin{defn}
A set $Z(y)$ is \textbf{monotonic} if for a fixed $y \in Y$, for all $z \in Z(y)$, if $0 \le \hat{z} \le z$, then $\hat{z} \in Z(y)$.
\end{defn}

Based on the constraints defined in Section \ref{modeling}, $Z(y)$ is monotonic. Intuitively, $\hat{z}$ can be thought of as a partial restructuring plan; choosing to restructure less arcs is still a feasible restructuring plan. If any $z_{ij} = 1$ but $\hat{z}_{ij} = 0$, then $\hat{z}$ is restructuring fewer arcs than $z$. Under this assumption, any constraints defining $Z(y)$ that are independent of $y$ will be satisfied by $\hat{z}$. As long as we can easily evaluate the dependency of the lower level decisions on the upper level decisions, we can find non-trivial partial information responding to new interdiction plans.

To derive the column-and-constraint generation algorithm, we first consider the maximization problem of the defender as two separate maximization problems, first maximizing over restructuring decisions $z$, then maximizing over flow decisions $x$. We can then take the dual of the inner-most problem (the flow problem) to offer an alternative formulation. Since the inner-most maximization problem will only have continuous variables, by strong duality of linear programs, the objective values of the bilevel and trilevel formulation would be equivalent. Let $\pi_i^+$ be the dual variable associated with the flow balance constraint for flow into node $i$, $\pi_i^-$ be the dual variable associated with the flow balance constraint for flow leaving node $i$, and $\theta_i$ and $\theta_{ij}$ be the dual variables associated with the capacity constraint for node $i$ and arc $(i,j)$, respectively. We now present the equivalent dual-based formulation of \eqref{model1}:

\begin{singlespace}
\small\begin{align}
    \label{tlforig}
    \min_y \max_z \min_{\pi,\theta} ~~~ & \sum_{i \in N\setminus \{s,t\}} u_{i}(1-y_i) \theta_i + \sum_{(i,j) \in A} u_{ij} \theta_{ij} \nonumber\\ &+ \sum_{(i,j) \in A^R} u_{ij} (z_{ij}^{in} + z_{ij}^{out}) \theta_{ij}  \nonumber\\  
    \text{s.t.} \hspace{.050cm }& \pi_j^+ + \theta_{sj} \ge 1 & \text{ for } (s,j) \in A \cup A^R \nonumber\\
    & \pi_j^+ - \pi_i^- + \theta_{ij} \ge 0 & \text{ for } (i,j) \in A \cup A^R \text{ s.t. } i \ne s, j \ne t  \nonumber\\
    & \pi_i^- - \pi_i^+ + \theta_{i} \ge 0 & \text{ for } i \in N\setminus \{s,t\} \\
    & -\pi_i^- + \theta_{it} \ge 0 & \text{ for } (i,t) \in A \cup A^R \nonumber\\
    & \theta \ge 0 \nonumber\\
    & z^{in}, z^{out} \in Z(y) \nonumber\\
    & y \in Y \nonumber
\end{align}
\end{singlespace}

We note that the dual to the max flow problem is the min cut problem. While there are many solutions to the linear program, we can identify when a solution corresponds to a feasible cut. 

\begin{defn}
We say a solution $(\pi^+, \pi^-, \theta)$ is a \textbf{feasible cut solution} if the following conditions are true:
\begin{itemize}
    \item If node $i$ is on the source side of the cut, $\pi_i^+ = \pi_i^- = 1$ and $\theta_i = 0$
    \item If a node $i$ is on the sink side of the cut, $\pi_i^+ = \pi_i^- = 0$ and $\theta_i = 0$
    \item If a node $i$ is in the cut, $\pi_i^+ = 1$, $\pi_i^- = 0$, and $\theta_i=1$
    \item If an arc $(i,j)$ is in the cut, $\theta_{ij} = 1$, and otherwise $\theta_{ij} = 0$
\end{itemize}
\end{defn}

In the classical minimum cut problem, the constraint matrix is totally unimodular, which implies that the continuous variables will be integral in the optimal solution to the linear program. Likewise, for fixed $(y, z)$, the inner-most minimization problem of \eqref{tlforig} will also satisfy this property.  Without loss of generality, any solution $(\pi^+, \pi^-)$ can be shifted to fit the above definition. Thus, there must be an optimal solution that is a feasible cut solution.

In a standard network interdiction problem, we would use the McCormick inequalities to linearize the $y_{i} \theta_{i}$ terms. However, in the context of our C\&CG algorithm, using the McCormick inequalities to linearize the $z_{ij}\theta_{ij}$ terms will result in a large number of variables and constraints being generated. We can instead use the following equivalent formulation of the dual problem, similar to Model 1D presented in \cite{orig}:

\begin{singlespace}
\small\begin{align}
    \label{tlf}
    \min_{y\in Y} \max_z \min_{\pi,\theta} ~~~ & \sum_{i \in N \setminus \{s,t\}} u_i \theta_i + \sum_{(i,j) \in A \cup A^R} u_{ij} \theta_{ij} \nonumber\\  
    \text{s.t.} \hspace{.050cm }& \pi_j^+ + \theta_{sj} \ge 1 & \text{ for } (s,j) \in A \nonumber\\
    & \pi_j^+ - \pi_i^- + \theta_{ij} \ge 0 & \text{ for } (i,j) \in A   \nonumber \text{ s.t. } i \ne s, j \ne t \nonumber\\
    & \pi_i^- - \pi_i^+ + \theta_{i} \ge - y_{i} & \text{ for } i \in N  \setminus \{s,t\} \nonumber\\
    & -\pi_i^- + \theta_{it} \ge 0 & \text{ for } (i,t) \in A \nonumber \\
    & \pi_j^+ + \theta_{sj} \ge z_{sj}^{in} & \text{ for } (s,j) \in A^R \\
    & \pi_j^+ - \pi_i^- + \theta_{ij} \ge z_{ij}^{out} - 1 & \text{ for } (i,j) \in A^R \text{ s.t. } i \ne s, j \ne t \nonumber\\
    & \pi_j^+ - \pi_i^- + \theta_{ij} \ge z_{ij}^{in} - 1 & \text{ for } (i,j) \in A^R \text{ s.t. } i \ne s, j \ne t \nonumber\\
    & \theta \ge 0 \nonumber\\
    & z^{in}, z^{out} \in Z(y) \nonumber\\
    & y \in Y \nonumber
\end{align}
\end{singlespace}

We note that in \eqref{tlf}, we have no $z^{out}_{sj}$, $z^{in}_{it}$, $z^{out}_{it}$ variables. This is because the source and sink nodes are not representative of participants in the network, and thus have no restructuring capabilities of their own, removing the need to define $z^{out}_{sj}$ and $z^{in}_{it}$. We have no $z^{out}_{it}$ because we assume drug dealer nodes cannot restructure to the sink node (i.e., they would be selling their product to themselves), and that recruitable drug users are already connected to the sink, not a drug dealer. We do still have $z^{in}_{sj}$ variables in the case where a node can be promoted to become a supplier.

Similarly to \eqref{tlforig}, for any feasible $(y, z)$, there will be an optimal solution to \eqref{tlf} where the continuous variables will be integral. This allows for the following theorem.

\begin{thrm}
\label{thrm:equiv}
The objective value of the optimal solution to \eqref{tlforig} is equivalent to that of the objective value of the optimal solution to \eqref{tlf}.
\end{thrm}
\begin{proof}
Let $y \in Y$ be a feasible interdiction plan and $z \in Z(y)$ be a feasible restructuring plan. We demonstrate that any feasible solution that corresponds to a feasible cut in one program can be converted to a feasible solution that corresponds to a feasible cut in the other.

We first consider a feasible solution $(\pi^+, \pi^-, \theta)$ in \eqref{tlforig} that corresponds to a feasible cut solution. We now construct a feasible solution $(\hat{\pi}^+,\hat{\pi}^-, \hat{\theta})$ to \eqref{tlf} with the same objective value. We set $\hat{\pi}^+ = \pi^+$ and $\hat{\pi}^+ = \pi^+$, and will construct $\hat{\theta}$ such that the objective value associated with $(\hat{\pi}^+,\hat{\pi}^-, \hat{\theta})$ will be the same as the objective value associated with $(\pi^+, \pi^- \theta)$. With these choices of $\hat{\pi}^+$ and $\hat{\pi}^-$, then $\hat{\theta}_{ij} = \theta_{ij}$ for all $(i,j) \in A$ will be feasible. Note that if $\theta_{i}=0$ for $i \in N$, then $\pi^+_i \le \pi_i^-$. In \eqref{tlforig}, we will have that $(1-y_{i})\theta_{i} = 0$ regardless of $y_{i}$. Likewise, in \eqref{tlf}, the corresponding constraint will evaluate to $\theta_{i} \ge - y_{i}$. Since $y_{i} \in \{0, 1\}$, then $- y_{i} \le 0$. Thus, $\theta_{i} = 0$ is feasible in \eqref{tlf}.

Now consider where $\theta_{i} = 1$ in \eqref{tlforig}. Note that if $\theta_{i}=1$, it must be that $\pi_i^- - \pi_i^+ = - 1$, else $\theta_{i} = 0$ will be feasible and the solution is not a feasible cut solution. In \eqref{tlf}, the corresponding constraint will evaluate to $-1 + \theta_{i} \ge - y_{i}$. We have two cases: $y_{i} = 0$ and $y_{i}=1$. If $y_{i}=0$, then $(1-y_{i})\theta_{ij} = 1$. In \eqref{tlf} the corresponding constraint further evaluates to $-1 + \theta_{i} \ge 0$, thus enforcing $\theta_{i} \ge 1$. Thus, we will have that $\hat{\theta}_{i} = 1 = (1 - y_{i})\theta_{i}$ is feasible. If $y_{i} = 1$, then $(1-y_{i})\theta_{i} = 0$. In \eqref{tlf}, the corresponding constraint evaluates to $-1 + \theta_{i} \ge -1$. Thus, $\theta_{i} \ge 0$ is feasible, and we will have that $\hat{\theta}_{i} = 0 = (1 - y_{i})\theta_i$. Thus, $\hat{\theta}_{i}$ in \eqref{tlf} will take the same value as $(1-y_{ij})\theta_{ij}$ in \eqref{tlforig}, resulting in the same objective values. Similar arguments can be used to construct feasible $\hat{\theta}_{ij}$ for $(i,j) \in A^R$ such that $z_{ij}\theta_{ij} = \hat{\theta}_{ij}$ for all $(i,j) \in A^R$. Since this can be done for any feasible $(y, z)$, any feasible solution in \eqref{tlforig} that corresponds to a feasible cut solution has an equivalent solution in \eqref{tlf}.

To convert from a solution of \eqref{tlf} to \eqref{tlforig}, we can follow the process in reverse. Let $(\hat{\pi}^+, \hat{\pi}_-, \hat{\theta})$ be a feasible solution to \eqref{tlf} that corresponds to a feasible cut solution. We again take $\pi^+ = \hat{\pi}^+$, $\pi^- = \hat{\pi}^-$, and $\theta_{ij} = \hat{\theta_{ij}}$ for $(i,j) \in A$. If $\hat{\theta}_{i} = 1$, we also have that $\theta_{i} = 1$ is feasible. Suppose $\hat{\theta}_{i} = 0$. We have two cases, $y_{i} = 0$ and $y_{i} = 1$. If $y_{i} = 0$, then $\hat{\pi}^-_i \ge \hat{\pi}_i^+$, and thus $\theta_{i} = 0$ is feasible. This choice results in $(1 - y_{i})\theta_{i} = 0 = \hat{\theta}_{i}$. If $y_{i} = 1$, we again have two cases. If $\hat{\pi}_i^- \ge \hat{\pi}_i^+$, then $\theta_{i} = 0$ is feasible. Otherwise, if $\hat{\pi}_i^- - \hat{\pi}_i^+ = - 1$, we must have $\theta_{i} \ge 1$ to maintain feasibility in \eqref{tlforig}. However, with $y_{i}=1$, $(1-y_{i})\theta_{i} = 0 = \hat{\theta}_{i}$. As before, similar arguments can be used to construct feasible $\theta_{ij}$ for $(i,j) \in A^R$ such that $z_{ij}\theta_{ij} = \hat{\theta}_{ij}$. Since this can be done for any feasible $(y, z)$, any feasible cut solution to \eqref{tlf} that corresponds to a cut solution has an equivalent solution in \eqref{tlforig}. 

For fixed $(y, z)$, we know that an optimal solution of the inner-most minimization problem will correspond to a feasible cut solution. Let $f_{(2)}^*$ be the optimal objective value of the inner-most minimization problem of \eqref{tlforig} and  $f_{(3)}^*$ be the optimal objective value of the inner-most optimization problem of \eqref{tlf}. Consider the optimal solution $(\pi^{+^*}, \pi^{-^*} \theta^*)$ of \eqref{tlforig} that is a feasible cut solution, with objective value $f_{(2)}^*$. We can convert this solution to a feasible cut solution of \eqref{tlf} $(\hat{\pi}^{+^{(2)}},\hat{\pi}^{-^{(2)}}, \hat{\theta}^{(2)})$, with equivalent objective value  $f_{(2)}^*$. Since this solution is feasible in \eqref{tlf}, we have that  $f_{(3)}^* \le  f_{(2)}^*$. Likewise, consider the optimal solution $(\hat{\pi}^{+^*},\hat{\pi}^{-^*}, \hat{\theta}^*)$ of \eqref{tlf} that is a feasible cut solution, with objective value $f_{(3)}^*$. We can also convert this solution to a feasible cut solution of \eqref{tlforig} $(\pi^{+^{(3)}}, \pi^{-^{(3)}}, \theta^{(3)})$, with equivalent objective value  $f_{(3)}^*$. Since this solution is feasible in \eqref{tlforig}, we have that $f_{(2)}^* \le  f_{(3)}^*$. Thus, for fixed $(y, z)$, we have that  $f_{(2)}^* =  f_{(3)}^*$. Since this holds regardless of choice of $(y, z)$, it must be that, for any $(y, z)$, the inner-most minimization problems of the two programs will have equivalent objective values. Thus, the objective values of \eqref{tlforig} and \eqref{tlf} will be equivalent.
\end{proof}

The difference between the two models is quite subtle. Where \eqref{tlforig} computes the cut in the original network, then determines that interdicted arcs should not be counted towards the value and restructured arcs should be counted towards the value of the cut (if they are in it), \eqref{tlf} determines exactly what the cut is in the network after interdictions and restructuring have occurred.

We now remark that there are a finite number of interdiction plans, so $Y$ is finite. Additionally, for a given interdiction plan $y$, $Z(y)$ is finite. Thus, $\bigcup_{y \in Y} Z(y)$ is finite. For a given $y \in Y$, we can replace the maximization problem by enumerating over all possible decisions $z \in Z(y)$, defining $\pi^z$ and $\theta^z$ to be the partition of the node set and indicator of which arcs are in the cut corresponding to the network $G(z) = (N, A \cup \{(i,j) \in A^R: z_{ij}=1\})$. We additionally make the objective value associated with each $z$ a constraint, and enforce that $\eta$, the objective value of the bilevel problem, is at least as large as the objective value associated with each $z$. The following formulation is non-standard, since the set of constraints depends on the decision $y$.
\vspace{-9mm}
\begin{singlespace}
\begin{align}
    \label{tlfFE}
    \min_{y, \pi^+, \pi^-,\theta, \eta} ~~~ & \eta \nonumber\\
    \text{s.t.} \hspace{.050cm } & \eta \ge \sum_{i \in N \setminus\{s,t\}} u_i \theta_i^z + \sum_{(i,j) \in A \cup A^R} u_{ij} \theta_{ij}^z  & \text{ for } z \in Z(y) \nonumber\\  
    &\pi_j^{+^z} + \theta_{sj}^z \ge 1 & \text{ for } (s,j) \in A, z \in Z(y) \nonumber\\
    & \pi_j^{+^z} - \pi_i^{-^z} + \theta_{ij}^z \ge 0 & \text{ for } (i,j) \in A \text{ s.t. } i \ne s, j \ne t, z \in Z(y)  \nonumber\\
    & \pi_i^{-^z} - \pi_i^{+^z} + \theta_{i}^z \ge - y_{i} & \text{ for } i \in N  \setminus\{s,t\}, z \in Z(y)  \nonumber\\
    & -\pi_i^{-^z} + \theta_{it}^z \ge 0 & \text{ for } (i,t) \in A, z \in Z(y)   \\
    & \pi_j^{+^z} + \theta_{sj}^z \ge z_{sj}^{in} & \text{ for } (s,j) \in A^R, z \in Z(y)  \nonumber\\
    & \pi_j^{+^z} - \pi_i^{-^z} + \theta_{ij}^z \ge z_{ij}^{out} - 1 & \text{ for } (i,j) \in A^R \text{ s.t. } i \ne s, j \ne t, z \in Z(y)   \nonumber\\
    & \pi_j^{+^z} - \pi_i^{-^z} + \theta_{ij}^z \ge z_{ij}^{in} - 1 & \text{ for } (i,j) \in A^R \text{ s.t. } i \ne s, j \ne t, z \in Z(y)   \nonumber\\
    & \theta \ge 0 \nonumber\\
    & y \in Y \nonumber
\end{align}
\end{singlespace}

In \eqref{tlfFE}, the restructuring decisions $z$ now become parameters, as we enumerate over all possible restructuring decisions, leaving the attacker as the only decision-maker. In terms of our problem, the method of Yue et al.\ \cite{cncg2} would convert this non-standard formulation to a standard formulation by including implication constraints to determine if a restructuring plan is feasible with respect to the current interdiction plan, and only including constraints associated with feasible restructuring plans. These implication constraints prevent infeasible solutions from being considered when determining a lower bound on the objective function, which would invalidate the lower bound. They then reduce the size of the problem by only considering a subset of restructuring plans, then iteratively adding new plans that are relevant to determining the true objective value. However, if making a small change to the interdiction plan results in making many of the restructuring plans infeasible, then the lower bound on the true objective value will grow at a slower rate, causing the algorithm to converge slowly.

Instead of removing constraints associated with infeasible restructuring plans, we exploit the monotonicity of $z$ to identify the components of each restructuring plan that are feasible after changing the interdiction decision. Whenever we can identify a non-trivial restructuring plan from an infeasible restructuring plan, we include an additional valid lower bound on the true objective value that the method of Yue et al.\ \cite{cncg2} would exclude. Figure \ref{fig:wEx} demonstrates the benefit of this inclusion.

Consider the network in Figure \ref{fig:wEx1}. Suppose the attacker has the budget to interdict at most two nodes in the bottom tier, any interdiction allows the parent node of an interdicted node to restructure to any one node in the bottom tier, and the defender can restructure at most two arcs. Suppose we know the interdiction plan and restructuring plan identified in Figure \ref{fig:wEx2}. If we were to change our interdiction plan to the plan identified in Figure \ref{fig:wEx4}, the dotted arc on the right is not allowed to be restructured, but the dashed arc on the left is still able to be restructured. The method of \cite{cncg2} would determine that, since one arc cannot be restructured, the entire restructuring plan must be discarded, as demonstrated in Figure \ref{fig:wEx4}. This method would tell us that the objective value associated with that interdiction plan would be at least $31$. However, if we instead consider the restructuring sub plan of just the dashed arc on the left, as demonstrated in Figure \ref{fig:wEx3}, we can observe that the objective value associated with that interdiction plan is higher than what the method of \cite{cncg2} would identify. Our method would identify that the objective value associated with that restructuring plan would be at least $34$. Because of this, applying their method will require more iterations to identify the optimal solution than our method will, since their method results in a slower increase in the lower bound on objective value.

\begin{figure}[h!]
\centering\begin{subfigure}{.45\textwidth}
  \centering
  \includegraphics[width=\linewidth]{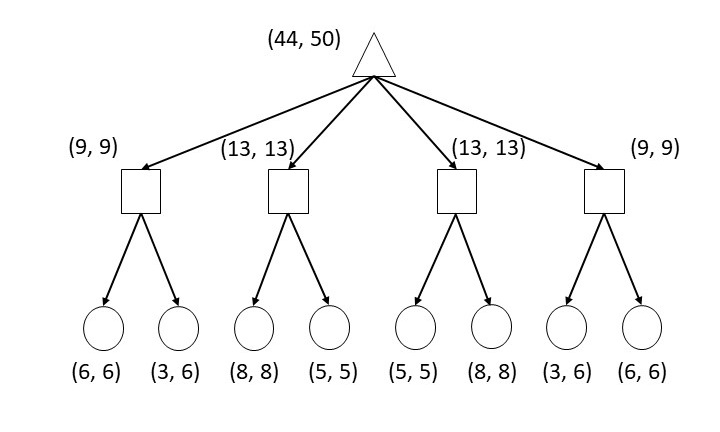}
  \caption{Example network with optimal flow}
  \label{fig:wEx1}
\end{subfigure}%
\begin{subfigure}{.45\textwidth}
  \centering
  \includegraphics[width=\linewidth]{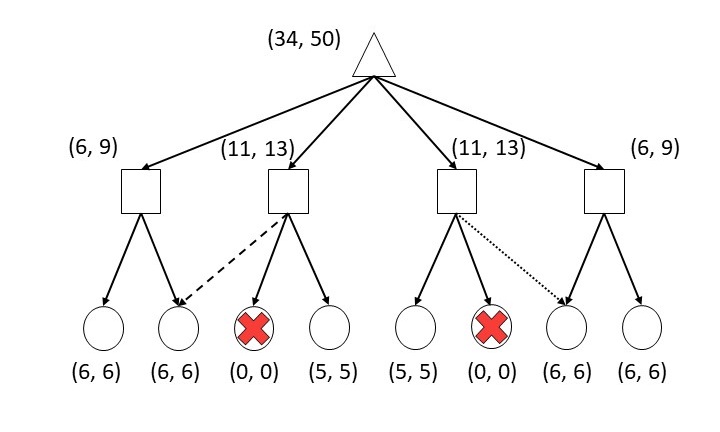}
  \caption{Sample interdiction and restructuring plans}
  \label{fig:wEx2}
\end{subfigure}

\begin{subfigure}{.45\textwidth}
  \centering
  \includegraphics[width=\linewidth]{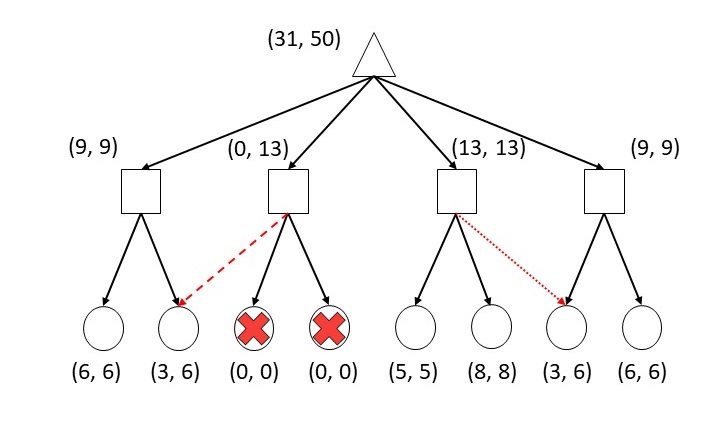}
  \caption{Restructuring plan discarded by \cite{cncg2} with new interdiction plan}
  \label{fig:wEx4}
\end{subfigure}%
\begin{subfigure}{.45\textwidth}
  \centering
  \includegraphics[width=\linewidth]{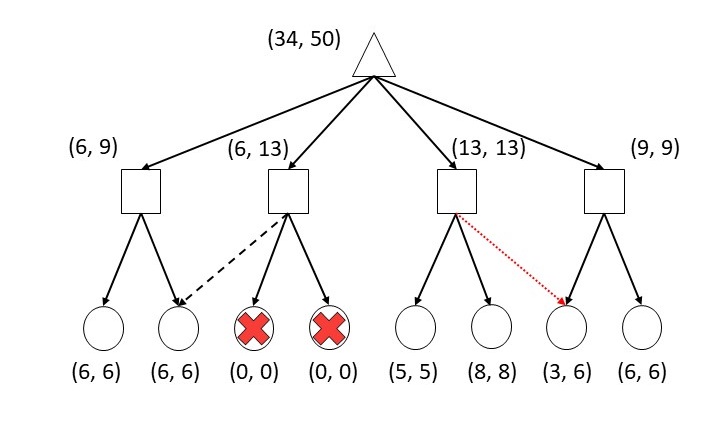}
  \caption{Feasible restructuring sub plan identified}
  \label{fig:wEx3}
\end{subfigure}%
\caption{Impact of identifying restructuring sub plans on determining the objective value of the MFNIP-R}
\label{fig:wEx}
\end{figure}

We now demonstrate how to identify feasible restructuring sub plans. Consider two interdiction plans $\bar{y}, \hat{y} \in Y$ where $\hat{y} \ne \bar{y}$. For a restructuring plan $\bar{z} \in Z(\bar{y}) \setminus Z(\hat{y})$ we wish to identify some point $\hat{z} \in Z(\hat{y})$ that is ``near" $\bar{z}$, i.e., we want to identify the restructuring plan that restructures as many arcs in $\bar{z}$ as possible while still being feasible with respect to the new interdiction plan $\hat{y}$. We define $w = (w^{in}, w^{out})$ to be the indicator of which arcs restructured in $\bar{z}$ are still able to restructured in response to the new interdiction plan $\hat{y}$. Referencing the network in Figure \ref{fig:wEx}, if $\bar{y}$ was the interdiction plan in Figure \ref{fig:wEx2} and $\bar{z}$ was the restructuring plan in Figure \ref{fig:wEx2}, then we would want to define constraints on $y$ and $w$ such that when we consider $\hat{y}$ to be the interdiction plan in Figure \ref{fig:wEx3}, we are able to use the $w$ variables to identify the restructuring plan in Figure \ref{fig:wEx3} as $\hat{z}$.

Since we have variables $z^{in}$ and $z^{out}$, we need to have variables $w^{in}$ and $w^{out}$ in our model. Because of this, we also need to adapt our model to reflect having two ways to restructure. Our constraints involving restructuring decisions are of the form $\pi^z_{j} - \pi^z_{i} + \theta^z_{ij} \ge z_{ij} - 1$. When $w_{ij}=1$, this does not consider whether $z_{ij}=1$ because $z^{out}_{ij}=1$ or $z^{in}_{ij}=1$. This consideration is needed in order to maintain feasibility with respect to the interdiction decisions $y$. If, for example, $z^{in}_{ij}=1$, but $w^{in}_{ij}=0$ and $w^{out}_{ij}=1$, we would not have that $(i, j)$ is a restructured arc in the sub plan, as no interdiction has occurred to allow this restructuring. To incorporate this, we need to replace the above set of constraints with a pair of constraints, $\pi^k_{j} - \pi^k_{i} + \theta^k_{ij} \ge z^{out,k}_{ij} + w^{out,k}_{ij} - 2$ and $\pi^k_{j} - \pi^k_{i} + \theta^k_{ij} \ge z^{in,k}_{ij} + w^{in,k}_{ij} - 2$. With this pair of constraints, the relationship between $\pi_{i}$, $\pi_{j}$ and $\theta_{ij}$ is only enforced when $z^{out,k}_{ij}=1$ and $w^{out,k}_{ij}=1$, or when $z^{in,k}_{ij}=1$ and $w^{in,k}_{ij}=1$. 

We now derive constraints based on equations \eqref{leave} - \eqref{budget} from Section \ref{modeling}. By monotonicity, any sub plan of a previously visited restructuring plan will satisfy constraints \eqref{leave}, \eqref{enter}, \eqref{noDouble}, and \eqref{budget}, since there is no dependency on $y$ in these constraints. What remains to enforce is feasibility based on constraints \eqref{neighborModAbove} and \eqref{neighborModBelow}. This means that we need define constraints to determine how many restructurings each node can perform based on a given interdiction plan, and which of the arcs chosen to be restructured in previously considered restructuring plans are still able to be restructured. However, we must also account for the fact that more interdictions can be done than potential allowable restructurings. To account for this, as well as the potential fractionality based on $q_i$, we define additional variables, $\delta_{i}^{in,k}, \delta_{i}^{out,k} \in \{0,1\}$, $\kappa_{i}^{in,k}, \kappa_{i}^{out,k} \in \mathbb{Z}_{\ge 0}$, and an additional parameter $\gamma > 0$. The variables $\kappa_{i}^{in,k}$ and $\kappa_{i}^{out,k}$ are representative of the number of arcs allowed to be restructured into and out of node $i$, respectively. The variables $\delta_{i}^{in,k}$ and $\delta_{i}^{out,k}$ are indicators of if, for each node $i$, every arc into or out of $i$ in the $k^{th}$ plan will be restructured, respectively. For each $k$, we can compute $w^{in,k}$ and $w^{out,k}$ with the following sets of constraints:

\begin{singlespace}
\begin{align}
    & q_j \sum_{h \in N:(h,j) \in A} y_{h} - 1 + \gamma \le \kappa_{j}^{in,k} \le q_j \sum_{h \in N:(h,j) \in A} y_{h} & \text{ for } j \in N \setminus \{s,t\}\\
    & q_j |\{h \in N:(h,j) \in A\}| \cdot \delta_{i}^{in,k} + \sum_{i \in N: (i,j)\in A^R, z^{in,k}_{ij} = 1} w^{in,k}_{ij} \ge \kappa_{j}^{in,k} & \text{ for } j \in N \setminus \{s,t\}\\
    & \delta_{j}^{in,k} \le \frac{\sum_{j \in N: (i,j)\in A^R, z^{in,k}_{ij} = 1} w^{in, k}_{ij}}{|\{i \in N: (i,j)\in A^R, z^{in,k}_{ij} = 1\}|} & \text { for } j \in N \setminus \{s,t\}\\
    & q_i \sum_{h \in N:(i,h) \in A} y_{h} - 1 + \gamma \le \kappa_{i}^{out,k} \le q_i \sum_{h \in N:(i,h) \in A} y_{h} & \text{ for } i \in N \setminus \{s,t\}\\
    & q_i |\{h \in N:(i,h) \in A\}| \cdot \delta_{i}^{out,k} +   \sum_{j \in N: (i,j)\in A^R, z^{out,k}_{ij} = 1} w^{out,k}_{ij} \ge \kappa_{i}^{out,k} & \text{ for } i \in N \setminus \{s,t\}\\
    & \delta_{i}^{out,k} \le \frac{\sum_{j \in N: (i,j)\in A^R, z^{out,k}_{ij} = 1} w^{out, k}_{ij}}{|\{j \in N: (i,j)\in A^R, z^{out,k}_{ij} = 1\}|} & \text { for } i \in N \setminus \{s,t\}
\end{align}
\end{singlespace}

Constraints (16) and (19) determine the number of arcs that are allowed to be restructured into or out of each node $j$, capturing the value of $\kappa_{j}^{in,k} = \lfloor q_j \sum_{h \in N:(h,j) \in A} y_{h} \rfloor$ and $\kappa_{j}^{out,k} = \lfloor q_j \sum_{h \in N:(j,h) \in A} y_{h}\rfloor$, respectively. Note that these sums in the inequalities are the same as the sums in \eqref{neighborModAbove} and \eqref{neighborModBelow}. Knowing how many arcs are allowed to restructure into or out of each node, constraints (17) and (20) determine which of the arcs were restructured in the $k^{th}$ restructuring plan should still be restructured, as determined by the attacker. However, we need to prevent infeasibility when the number of arcs that \textit{can be} restructured from a given node is more than the number of arcs that \textit{were} restructured from that node. If the number of arcs that can be restructured into (or out of) a node $j$ is greater than the number of arcs that have been restructured into (or out of) that node in the $k^{th}$ plan, $\delta_{j}^{in,k} = 1$ (or $\delta_{j}^{out,k}=1$, respectively), allowing the constraint to maintain feasibility. Constraints (18) and (21) enforce that these indicators are $0$ unless every previously restructured arc into (or out of) each node has been determined to also be restructured in a new plan. If $q_j$ is integral, the constraints (16) and (19) are unnecessary, since the inequalities in the constraint will never be able to take fractional values. In this case, $q_j \sum_{l:(l,j) \in A} y_{l}$ can be directly substituted in for $\kappa_{j}^{in,k}$ and $\kappa_{j}^{out,k}$.

With the inclusion of these constraints, we can now convert \eqref{tlfFE} to a standard formulation, where we enumerate over the fixed set $\bigcup_{y \in Y} Z(y)$ instead of just $Z(y)$. Let $\left|\bigcup_{y \in Y} Z(y) \right| = M$. The following is a standard formulation of \eqref{tlfFE}:

\begin{singlespace}
\footnotesize\begin{align}
    \label{tlfFEwP}
    \min_{y, \pi^+, \pi^-,\theta, \eta} ~~~ & \eta \nonumber\\
    \text{s.t.} \hspace{.050cm } & \eta \ge \sum_{i \in N \setminus \{s,t\}} u_i \theta_i^k + \sum_{(i,j) \in A \cup A^R} u_{ij} \theta_{ij}^k  & \text{ for } k=1,\ldots,M \nonumber\\  
    &\pi_j^{+^k} + \theta_{sj}^k \ge 1 & \text{ for } (s,j) \in A, k=1,\ldots,M \nonumber\\
    & \pi_j^{+^k} - \pi_i^{-^k} + \theta_{ij}^k \ge 0 & \text{ for } (i,j) \in A \text{ s.t. } i \ne s, j \ne t, k=1,\ldots,M \nonumber\\
    & \pi_i^{-^k} - \pi_i^{+^k} + \theta_{i}^k \ge - y_{i} & \text{ for } i \in N \setminus \{s,t\}, k=1,\ldots,M \nonumber\\
    & -\pi_i^{-^k} + \theta_{it}^k \ge 0 & \text{ for } (i,t) \in A, k=1,\ldots,M  \nonumber\\
    & \pi_j^{+^k} + \theta_{sj}^k \ge z_{sj}^{in,k} + w_{sj}^{in,k} - 1 & \text{ for } (s,j) \in A^R, k=1,\ldots,M\\
    & \pi_j^{+^k} - \pi_i^{-^k} + \theta_{ij}^k \ge z_{ij}^{out,k} + w_{ij}^{out,k} - 2 & \text{ for } (i,j) \in A^R\text{ s.t. } i \ne s, j \ne t, k=1,\ldots,M \nonumber\\
    & \pi_j^{+^k} - \pi_i^{-^k} + \theta_{ij}^k \ge z_{ij}^{in,k} + w_{ij}^{in,k} - 2 & \text{ for } (i,j) \in A^R\text{ s.t. } i \ne s, j \ne t, k=1,\ldots,M \nonumber\\
    & (16) - (21) & \text { for } i \in N\setminus \{s,t\}, k=1,\ldots,M \nonumber\\
    & w^k \in \{0, 1\}^{|A^R|} & \text{ for } k=1,\ldots,M \nonumber\\
    & \theta \ge 0 \nonumber\\
    & y \in Y \nonumber
\end{align}
\end{singlespace}

Now that we have a standard formulation enumerating over fixed set $\bigcup_{y \in Y} Z(y)$, we instead seek to identify a valid lower bound by enumerating over a subset of point in this set. Let $\{z^1, \ldots, z^n\} \subseteq \bigcup_{y \in Y} Z(y)$, where $n \le M$. By enumerating over these points, we can formulate a program that provides a valid lower bound on \eqref{tlfFEwP}.

\begin{singlespace}
\footnotesize\begin{align}
    \label{tlfPEwP}
    \min_{y, \pi^+, \pi^-,\theta, \eta} ~~~ & \eta \nonumber\\
    \text{s.t.} \hspace{.050cm } & \eta \ge \sum_{i \in N \setminus \{s,t\}} u_i \theta_i^k + \sum_{(i,j) \in A \cup A^R} u_{ij} \theta_{ij}^k  & \text{ for } k=1,\ldots,n \nonumber\\  
    &\pi_j^{+^k} + \theta_{sj}^k \ge 1 & \text{ for } (s,j) \in A, k=1,\ldots,n \nonumber\\
    & \pi_j^{+^k} - \pi_i^{-^k} + \theta_{ij}^k \ge 0 & \text{ for } (i,j) \in A \text{ s.t. } i \ne s, j \ne t, k=1,\ldots,n \nonumber\\
    & \pi_i^{-^k} - \pi_i^{+^k} + \theta_{i}^k \ge - y_{i} & \text{ for } i \in N \setminus \{s,t\}, k=1,\ldots,n \nonumber\\
    & -\pi_i^{-^k} + \theta_{it}^k \ge 0 & \text{ for } (i,t) \in A, k=1,\ldots,n  \nonumber\\
    & \pi_j^{+^k} + \theta_{sj}^k \ge z_{sj}^{in,k} + w_{sj}^{in,k} - 1 & \text{ for } (s,j) \in A^R, k=1,\ldots,n\\
    & \pi_j^{+^k} - \pi_i^{-^k} + \theta_{ij}^k \ge z_{ij}^{out,k} + w_{ij}^{out,k} - 2 & \text{ for } (i,j) \in A^R\text{ s.t. } i \ne s, j \ne t, k=1,\ldots,n \nonumber\\
    & \pi_j^{+^k} - \pi_i^{-^k} + \theta_{ij}^k \ge z_{ij}^{in,k} + w_{ij}^{in,k} - 2 & \text{ for } (i,j) \in A^R\text{ s.t. } i \ne s, j \ne t, k=1,\ldots,n \nonumber\\
    & (16) - (21) & \text { for } i \in N\setminus \{s,t\}, k=1,\ldots,n \nonumber\\
    & w^k \in \{0, 1\}^{|A^R|} & \text{ for } k=1,\ldots,n \nonumber\\
    & \theta \ge 0 \nonumber\\
    & y \in Y \nonumber
\end{align}
\end{singlespace}

\begin{thrm}
\label{thrm:LB}
The objective value of the optimal solution to \eqref{tlfPEwP} is a lower bound on the objective value of \eqref{tlfFEwP}. 
\end{thrm}
\begin{proof}
\sloppy Let $y \in Y$ and $\{z^1, \ldots, z^{n}\} \subseteq \bigcup_{y \in Y} Z(y)$, and consider a feasible solution to \eqref{tlfFEwP}, $(\hat{\eta},\hat{y},\hat{\pi}^+,\hat{\pi}^-,\hat{\theta})$. Since $\{z^1, \ldots, z^{n}\} \subseteq \bigcup_{y \in Y}Z(y)$, there are variables $(\hat{\pi}^{+^{z^1}}, \ldots, \hat{\pi}^{+^{z^{n}}}, \hat{\pi}^{-^{z^1}}, \ldots, \hat{\pi}^{-^{z^{n}}}, \hat{\theta}^{z^1}, \ldots, \hat{\theta}^{z^{n}})$ associated with $\{z^1, \ldots, z^{n}\}$. The solution $(\hat{\eta},\hat{y},\hat{\pi}^{+^1}, \ldots, \hat{\pi}^{+^{n}}, \hat{\pi}^{-^1}, \ldots, \hat{\pi}^{-^{n}}, \hat{\theta}^1, \ldots, \hat{\theta}^{n})$ is a feasible solution in \eqref{tlfPEwP}. Since this solution is feasible in \eqref{tlfPEwP}, we know $\eta^* \le \hat{\eta}$. Thus, a feasible solution to \eqref{tlfPEwP} can be created from any solution to \eqref{tlfFEwP}. Thus, the objective value of \eqref{tlfPEwP} is a lower bound on the objective value of \eqref{tlfFEwP}.
\end{proof}

Once we have solved \eqref{tlfPEwP} for optimal first stage decision $\hat{y}$, we do not yet know if any $z^k$ is the optimal response to $\hat{y}$. We can easily compute the maximization decisions $\hat{z}$ that are the optimal response to $\hat{y}$. We do so by solving the following:
\vspace{-8mm}
\begin{singlespace}
\begin{align}
    \label{OptResponseSP}
    \max_{x, z} ~~~ & \sum_{(s,j) \in A \cup A^R} x_{sj} \nonumber\\  
    \text{s.t.} \hspace{.050cm }& \sum_{(j, i) \in A \cup A^R} x_{ji} -  x_{i} = 0 & \text{ for } i \in N \setminus \{s,t\} \nonumber\\
    & x_i - \sum_{(i, j) \in A \cup A^R} x_{ij} = 0 & \text{ for } i \in N \setminus \{s,t\} \nonumber\\
    & 0 \leq x_{ij} \le u_{ij} & \text{ for } (i,j) \in A\\
    & 0 \leq x_{i} \le u_i (1 - \hat{y}_{i}) & \text{ for } i \in N \setminus \{s,t\} \nonumber\\
    & 0 \leq x_{ij} \le u_{ij}z_{ij} & \text{ for } (i,j) \in A^R \nonumber \\
    & z \in Z(\hat{y}) \nonumber
\end{align}
\end{singlespace}

When we solve this problem, we recover a bilevel feasible solution. Thus, the optimal objective value of this program is an upper bound on the true solution. Further, we can introduce its optimal solution $\hat{z}$ as a new point in the partial enumeration. Thus, we can iteratively find upper bounds $U^k$ and lower bounds $L^k$ on the true objective value of \eqref{model1}. We stop the algorithm when $U^k - L^k \le \epsilon$, where $\epsilon \ge 0$. If we choose $\epsilon = 0$, we solve the problem exactly. We now introduce the column and constraint generation procedure for MFNIP-R as Algorithm \ref{alg:cncg}.

\begin{algorithm}
\caption{C\&CG for MFNIP-R}
\label{alg:cncg}
\begin{algorithmic}
\STATE{\textbf{Initialize:}} lower bound $L = - \infty$, upper bound $U = \infty$, optimal interdiction decisions $y^*=0$, optimal restructuring decisions $z^1 = z^* = 0$, error tolerance $\epsilon \ge 0$. iteration counter $n=1$.
 \WHILE{$U - L > \epsilon$}
\STATE{\textbf{Step 1.}} Solve \eqref{tlfPEwP} for optimal interdiction decision $y^n$ and objective value $\eta^n$. Set $L = \eta^n$.
\IF{$U - L \le \epsilon$}
 \STATE{terminate; $y^*$ and $z^*$ are the optimal decisions with objective value $U$.}
\ENDIF
\STATE{\textbf{Step 2.}} Input $y^n$ into \eqref{OptResponseSP} as data and solve \eqref{OptResponseSP} for optimal restructuring decisions $z^{n+1}$ and objective value $\sum_{(s,j) \in A \cup A^R} x_{sj}$.
\IF{$x_{ts} < U$}
 \STATE{Let $y^* = y^n$, $z^* = z^{n+1}$, $U = \sum_{(s,j) \in A \cup A^R} x_{sj}$}.
\ENDIF
\STATE{\textbf{Step 3.}} Include constraints corresponding to $z^{n+1}$ in \eqref{tlfPEwP}, create variables $\pi^{n+1}$, $\theta^{n+1}$, $w^{n+1}$, set $n = n + 1$, return to Step 1. 
\ENDWHILE
\end{algorithmic}
\end{algorithm}

\begin{thrm}
\label{thrm:correct}
Algorithm \ref{alg:cncg} converges to the optimal solution of \eqref{tlfFE} when $\epsilon=0$.
\end{thrm}

\begin{proof}
First, observe that if $\epsilon=0$, we need $L\ge U$ for Algorithm \ref{alg:cncg} to terminate, meaning we need $\eta^n \ge \min\{\hat{\eta}^1, \ldots, \hat{\eta}^{n-1}\}$. Suppose that a solution $\hat{y}$ that is optimal in \eqref{tlfPEwP} is visited twice, at iterations $i$ and $j$, where $i < j$. The algorithm will terminate, since, by the first set of constraints in \eqref{tlfPEwP}, $\eta^j = \hat{\eta}^i \ge \min\{\hat{\eta}^1, \ldots, \hat{\eta}^{n-1}\}$. Let $y^*$ be the optimal solution to the full enumeration problem with objective value $\eta^*$, and suppose that Algorithm \ref{alg:cncg} never visits $y^*$. First, consider the case where the algorithm terminates with a solution $\hat{y}$ such that $\hat{\eta} < \eta^*$. Since the algorithm terminated, we found a bilevel feasible solution $(\hat{y}, \hat{z})$ with objective value $\hat{\eta}$. Thus $y^*$ is not the optimal solution, creating a contradiction. Now consider the case where the algorithm terminates with a solution with objective value $\hat{\eta} > \eta^*$. Then there must be some iteration $k$ such that $\eta^k > \eta^*$. By Theorem \ref{thrm:LB}, we can construct a corresponding solution to \eqref{tlfPEwP} from the optimal solution to \eqref{tlfFEwP}. Since this solution is feasible in \eqref{tlfPEwP}, $\eta^* \ge \eta^k$, creating a contradiction. Thus, the optimal solution must be visited by Algorithm \eqref{alg:cncg}. Since the sets $Y$ and $\bigcup_{y \in Y} Z(y)$ are finite sets, we know that the optimal solution will be reached in finite time.
\end{proof}

From Step $3$ of Algorithm \ref{alg:cncg}, we can revisit why we choose to derive the problem from \eqref{tlf} instead of \eqref{tlforig}. If we were to use \eqref{tlforig}, we would have bilinearities in $y_{i} \theta^k_{i}$ for $i \in N$ for every iteration $k$. Additionally, including $w^{in,k}$ and $w^{out,k}$ in this model would result in more bilinear terms $z_{ij}^{in,k} w^{in,k}_{ij} \theta^k_{ij}$ and $z_{ij}^{out,k} w^{out,k}_{ij} \theta^k_{ij}$ for all $(i,j) \in A^R$ for every iteration $k$. Each of these bilinearities would need to be linearized with the McCormick inequalities, resulting in $|N|+2|A^R|$ new variables and $O(|N|+2|A^R|)$ new constraints each iteration. These additional variables and constraints can cause the solve time of each iteration to take considerably longer due to the larger problem size.

We note that our idea for utilizing partial information can be modified to fit into other general BIP solvers. The main idea is understanding how previously visited lower level solutions impact the objective value while exploring upper level solutions. As long as we can define linear constraints to determine what integer components of a lower level solution remain feasible, and appropriately adjust the continuous components to maintain feasibility, when changing the upper level solution, we can identify non-trivial feasible lower level solutions that provide a better lower bound on the true objective value. This would most easily be done when $Z(y)$ is monotonic, as a feasible solution can be constructed by setting infeasible non-zero variables to be $0$. Upon reaching a bilevel feasible solution $(y,z)$, we can add new constraints to the problem determining what a non-trivial sub plan $\hat{z}$ with respect to other upper level decisions $\hat{y}$, updating each non-fathomed node in the branch-and-bound tree. Doing so will allow these solvers to better account for the impact of the lower level problem on the objective value of the BIP as the branching procedure occurs.

\subsection{Special Case for Constraints on $w$}
\label{ssec:special}
In the case where the maximum number of restructuring arcs leaving and entering $i$ is one (i.e., $l_i=s_i=1$) for all $i \in N$, we can define a different set of constraints to capture restructuring. Instead of defining $w^z$ based on when an arc can appear in each given restructuring plan, we can define $w_{ij}$ as the indicator of when an arc $(i,j)$ can be restructured based on the current interdiction plan. This means that when a given node $i$ is interdicted, every arc $(u,v)$ that can be restructured due to this interdiction will have $w^{in}_{uv}= 1$ or $w^{out}_{uv} = 1$ as appropriate. Since, for each previous plan, at most one of each arcs into (and out of) each node will have been restructured, at most one arc into (and out of) each node will have both $z_{uv}^{in,k} = 1$ and $w_{uv}^{in} = 1$ (or $z_{uv}^{out,k} = w_{uv}^{out} = 1$, respectively), guaranteeing feasibility. This eliminates the need for $\delta$. Additionally, since we can only restructure one arc into and out of each node, we no longer need to accurately capture how many arcs can be restructured for each node, but instead whether the node as the ability to restructure or not. This allows us to also eliminate $\kappa$. Suppose there exist integers $a_i, b_i$ such that $q_i = \frac{a_i}{b_i}$ for each node $i$. We can instead define our constraints on $w^{in}$ and $w^{out}$ as the following:
\begin{singlespace}
\begin{align}
    a_i |\{l:(l,j) \in A\}| w^{in}_{ij} \ge a_i \sum_{l:(l,j) \in A} y_{l} - b_i + 1  &\text{ for } (i,j) \in A^R \\
    a_i |\{l:(l,j) \in A\}| w^{out}_{ij} \ge a_i \sum_{l:(l,j) \in A} y_{l} - b_i + 1  &\text{ for } (i,j) \in A^R 
\end{align}
\end{singlespace}
Note that, if $q_i \ge 1$, we could replace $a_i$ and $b_i$ with $1$, as only one interdiction is needed to allow for the one allowed restructuring. Otherwise, when $q_i < 1$, $w^{in}_{ij}$ will only be $1$ when $a_i \sum_{l:(l,j) \in A} y_{l} \ge b_i$, which only occurs when $q_i \sum_{l:(l,j) \in A} y_{l} \ge 1$. This formulation allows us to generate all needed $w$ variables upfront, as opposed to generating more with each newly identified restructuring plan. In this case, constraints (25) and (26) would be inserted into \eqref{tlfFEwP} in place of constraints (16) - (21).

\section{Computational Results: Overview and Comparison to Other C\&CG Methods}
\label{compare}
We test our algorithm on five generated networks modeling city-level drug trafficking networks with $200$ users, over $10$ different attacker budgets and a fixed defender budget. These networks are generated following the procedure of Malaviya et al.\ \cite{multi}. The minimum budget tested allows for interdicting at least one node in the second highest tier for four of the five data sets, which is satisfied by the second smallest budget. These networks were generated following the procedure outlined in \cite{multi}, each consisting of six tiers. Restructurable arcs between tiers are created by generating the lists of arcs between each pair of sequential tiers, only using nodes currently in the network, then randomly selecting from each list. This guarantees that each pair of sequential tiers has restructurable arcs between them. This constitutes our base data sets. Table \ref{tbl:size} presents the size of these base data sets.

\begin{table}[h]
\begin{center}
\begin{tabular}{|c|c|c|c|}
\hline
Network & $|N|$ & $|A|$ & $|A^R|$ \\ \hline
1 & 262 & 904 & 128 \\ \hline
2 & 260 & 916 & 130 \\ \hline
3 & 261 & 917 & 131 \\ \hline
4 & 261 & 886 & 126 \\ \hline
5 & 262 & 899 & 127 \\ \hline
\end{tabular}
\caption{Size of base data sets}
\label{tbl:size}
\end{center}
\end{table}

From these base sets, we provide two adaptations. The first adaptation is the inclusion of recruitable participants, members who are not yet in the network but can be brought in to the network. The number of ``recruitable" users, dealers and safe houses is $20\%$ of the number of existing users, dealers and safe houses, respectively. These nodes are connected to the lower tiers of the network similar to \cite{multi}. All ``recruitable" users and dealers are randomly connected to nodes in the tier above them. ``Recruitable" safe house nodes are assigned arcs and restructurable arcs based on an existing safe house, to act as a backup should the existing safe house be interdicted. These generation methods also guarantee that each of these ``recruitable" participants has a path from the source through arc restructuring. The second adaptation we call organizational restructuring. This allows for the restructuring of arcs that skip a tier, emulating the promotion of a participant to a higher-level within the network. It also includes the recruitment of participants from other organizations. For our experiments, we set $k_i=1$ for all $i \in N$, $l_i=1$ for all $i \in N$, $a_{ij}=1$ for all $(i,j) \in A^R$, $l=1$ and $r=6$, allowing us to use the constraint constructed in Section \ref{ssec:special}. 

For the master problem, there will be approximately $|N|+2k|A^R|$ binary variables, $(2|N|+|A|+2|A^R|)k$ continuous variables, and $(|N|+|A|+4|A^R|)k$ constraints, where $k$ is the number of restructuring plans being considered. We note that, by the nature of the problem structure, there will be $k$ copies of the minimum cut problem in the master problem, one for each restructuring plan being considered. Future work could explore how to disaggregate the problem of determining the optimal interdiction decisions and corresponding projected restructuring decisions from the $k$ disjoint minimum cut problems in order to improve solution times. For the sub-problem, there will be approximately $2|A^R|$ binary variables, $|N|+|A|+|A^R|$ continuous variables, and $2|N|$ constraints.

Climbing the ladder constraints are included to enforce that interdictions start at nodes lower in the network before nodes higher in the network can be interdicted, similar to how law enforcement would investigate a drug trafficking network. The number of interdicted neighbors for the climbing the ladder constraints is determined as per Malaviya et al.\ \cite{multi}. The cost for interdicting a node in the climbing the ladder constraint is determined by aggregating the targeting requirement and arrest requirement values from Malaviya et al.\ \cite{multi}. We again note that interdicting is not solely law enforcement arresting participants, but can be any sort of intervention by an interested stakeholder. We can more appropriately think of the targetting and arrest requirements as the cost of observing participants and implementing interdictions, respectively. 

We limit the solve time of each experiment to two hours, with $\epsilon = 10^{-4}$. If an instance were to exceed the two hour solve time, we use the last iteration completed within the time limit to determine the lower and upper bounds. Additionally, after completing an iteration, we check if the total solve time plus the solve time of the most recent iteration is under two hours, and terminate if it is not. Since the problem grows in size each iteration, we expect that each subsequent iteration will be at least as long as the previous one, thus terminating early if we expect the next iteration to terminate after the time limit would be exceeded. We model the problem in AMPL with Gurobi 9.0.2 as the solver. Experiments were conducted on a laptop with an Intel\textsuperscript{\textregistered} Core\textsuperscript{TM} i5-8250 CPU @ 1.6 GHz - 1.8 GHz and 16 GB RAM running Windows 10.

We first compare solve time with the method described in \cite{cncg2}. Before doing so, we note that the KKT-conditions-based tightening cannot be applied to our problem. This tightening tries to improve the lower bound on the objective value of the upper level problem by having the upper level decision maker (attacker) try to predict the lower level decision maker's (defender's) integer decisions (restructuring decisions), and identifies the optimal continuous variables in the lower level problem. In our problem, for any interdiction plan, a restructuring plan of $z = 0$ is always feasible. Since we are in a minimization problem, the null plan will always be selected for these constraints, resulting in the lower bound given by the KKT-conditions-based tightening being the interdicted flow. If we were to enforce that restructuring needs to occur, we may encounter the scenario where an interdiction plan does not allow for restructuring to occur; this would make the master problem infeasible.

To implement their projection scheme, we extend our $w$ variables to determine the full feasibility of a restructuring plan. For each iteration $k$, let 
\vspace{-8mm}
\begin{singlespace}
\begin{equation*}
f^k = \begin{cases} 1 \text{ if the } k^{th} \text{ restructuring plan is feasible for current interdiction decision } y \\
0 \text{ otherwise}.
\end{cases}
\end{equation*}
\end{singlespace}
We can enforce that $f^k$ takes the correct value with the following constraint:
\vspace{-8mm}
\begin{singlespace}
\begin{align}
    f^k &\ge \sum_{(i,j) \in A^R: z^{out,k}_{ij} = 1} w_{ij}^{out,k} - |\{(i,j) \in A^R: z^{out,k}_{ij} = 1\}| \\
    &+ \sum_{(i,j) \in A^R: z^{in,k}_{ij} = 1} w_{ij}^{in,k} - |\{(i,j) \in A^R: z^{in,k}_{ij} = 1\}| + 1\text{.} \nonumber
\end{align}
\end{singlespace}
With this constraint, if all arcs that were restructured in the $k^{th}$ restructuring plan are able to be restructured, then the sums well be equivalent to the set cardinalities, and thus $f^k \ge 1$. Otherwise, the constraint reduces to $f^k \ge c$ with $c \le 0$, resulting in $f^k = 0$ being selected due to solving a minimization problem. Additionally, instead of $w^{in}_{ij}$ or $w^{out}_{ij}$ being included in the constraints to indicate the side of the cut each node is on, $f^k$ is included in its place.

Over the $250$ instances, our method only fails to solve three instances to optimality within two hours. All three of these instances are instances with the recruitment of new participants and have medium attacker budget levels. For simplicity of comparison, we compare our method against the algorithm of \cite{cncg2} for just the base instances and the instances with recruitment of new participants.

For the base instances, we observe that, when the algorithm of \cite{cncg2} is able to solve the problem, it typically requires more plans than Algorithm \ref{alg:cncg} to find the optimal solution. However, when both methods require a similar number of plans needed to solve the problem, our method is generally slower than that of the algorithm of \cite{cncg2}. We suspect this is due to our method needing to explore branch-and-bound nodes where parts of previously visited restructuring plans may be relevant. However, when our method needs less plans to find the optimal solution, we are able to solve the problem significantly faster (e.g., for the majority of instances for Data3, we require well under 20 minutes while their method hits the 2 hour upper bound). Table \ref{tbl:compareTimeBase} reports the run times and number of plans visited by each method.

\begin{table}[]
\begin{center}
\resizebox{\textwidth}{!}{
\begin{tabular}{|c|c|c|c|c|c|c|c|c|c|c|}
\hline
Budget & Data1, Alg1          & Data1, \cite{cncg2}     & Data2, Alg1          & Data2, \cite{cncg2}             & Data3, Alg1            & Data3, \cite{cncg2}    & Data4, Alg1          & Data4, \cite{cncg2}              & Data5, Alg1          & Data5, \cite{cncg2}       \\ \hline\hline
50     & \textbf{110.875 (3)} & 561.094 (6)          & \textbf{10.938 (3)}  & 73.079 (7)          & \textbf{25.236 (3)}    & 7200 (18)* & \textbf{18.86 (3)}   & 103.125 (6)          & \textbf{139.688 (3)} & 1282.75 (8)   \\ \hline
60     & \textbf{18.844 (3)}  & 55.328 (5)           & \textbf{5.672 (2)}   & 38.75 (4)           & \textbf{521.078 (8)}   & 7200 (17)* & \textbf{23.75 (3)}   & 399.781 (10)         & \textbf{70.422 (2)}  & 6132.687 (14) \\ \hline
70     & \textbf{12.781 (2)}  & 115.922 (5)          & \textbf{22.968 (3)}  & 181.234 (7)         & \textbf{449.516 (7)}   & 7200 (17)* & \textbf{51.765 (4)}  & 260.844 (8)          & \textbf{177.031 (3)} & 2139.485 (8)  \\ \hline
80     & 48.313 (3)           & \textbf{43.937 (3)}  & 47.078 (3)           & \textbf{41.594 (3)} & \textbf{1316.953 (7)}  & 7200 (15)* & \textbf{104.438 (4)} & 2388.859 (12)        & \textbf{170.422 (3)} & 7200 (14)*    \\ \hline
90     & \textbf{45.203 (2)}  & 118.797 (3)          & \textbf{122.359 (3)} & 155.86 (5)          & \textbf{688.641 (5)}   & 7200 (13)* & \textbf{139.515 (4)} & 533.484 (7)          & \textbf{56.063 (2)}  & 2843.36 (11)  \\ \hline
100    & 61.188 (2)           & \textbf{33.953 (2)}  & 95.031 (3)           & \textbf{75.468 (4)} & \textbf{69.656 (2)}    & 7200 (12)* & 722.641 (5)          & \textbf{282.453 (5)} & \textbf{210.328 (3)} & 7200 (9)*     \\ \hline
110    & \textbf{80.078 (2)}  & 179.234 (3)          & \textbf{3.219 (2)}   & 4.438 (3)           & \textbf{118.578 (2)}   & 7200 (11)* & \textbf{409.906 (4)} & 569.954 (5)          & \textbf{312.594 (3)} & 7200 (11)*    \\ \hline
120    & 378.562 (3)          & \textbf{240.813 (3)} & 0.156 (1)            & \textbf{0.141 (1)}  & \textbf{90.485 (2)}    & 7200 (12)* & 968.485 (6)          & \textbf{910.281 (7)} & \textbf{97.468 (2)}  & 4206.719 (10) \\ \hline
130    & 636.907 (3)          & \textbf{329.656 (3)} & \textbf{0.219 (1)}   & 0.578 (2)           & \textbf{1893.594 (7)}  & 7200 (13)* & \textbf{537.984 (6)} & 884.829 (7)          & \textbf{219.719 (3)} & 7200 (12)*    \\ \hline
140    & 14.343 (2)           & \textbf{3.813 (2)}   & 0.141 (1)            & \textbf{0.094 (1)}  & \textbf{4832.516 (12)} & 7200 (13)* & 62.437 (4)           & \textbf{9.546 (3)}   & \textbf{758.61 (5)}  & 3095.734 (9)  \\ \hline
\end{tabular}}
\end{center}
\caption{Comparison of run times (seconds) and number of plans visited of Algorithm \ref{alg:cncg} and the algorithm of \cite{cncg2} in base restructuring instances}
\label{tbl:compareTimeBase}
\end{table}

For the instances when recruitment of outside participants is allowed, we are able to more clearly see the advantage in our method over that of the algorithm of \cite{cncg2}. While our method is able to solve all but three instances, their method is only able to solve $9$ of the $50$ instances within the specified time limit.  We further observe that an extremely large percentage (38 out of 41) of our run times are under an hour and a large percentage (33 out of 41) are under half an hour for instances where they fail to solve the problem to optimality. Additionally, their method visits a significantly larger number of plans while still leaving a large relative optimality gap. Table \ref{tbl:compareTimeRec} compares the run times and number of plans visited by both methods. These instances are marked with an asterisk in Table \ref{tbl:compareTimeRec}. We additionally report the relative optimality gap of these instances as the difference between the lower and upper bound divided by the upper bound in Table \ref{tbl:optGapRec}.  These results demonstrate the significant importance of our cuts that use partial information about previous interdiction plans and the restructuring response.   

\begin{table}[]
\begin{center}
\resizebox{\textwidth}{!}{
\begin{tabular}{|c|c|c|c|c|c|c|c|c|c|c|}
\hline
Budget & Data1, Alg1          & Data1, \cite{cncg2}     & Data2, Alg1          & Data2, \cite{cncg2}             & Data3, Alg1            & Data3, \cite{cncg2}    & Data4, Alg1          & Data4, \cite{cncg2}              & Data5, Alg1          & Data5, \cite{cncg2}       \\ \hline\hline
50     & \textbf{55.5 (3)}     & 7200 (18)*          & \textbf{37.75 (4)}   & 7200 (35)*        & \textbf{211.891 (6)}  & 7200 (12)*   & \textbf{70.407 (4)}   & 7200 (29)*    & \textbf{161.484 (5)}   & 7200 (17)* \\ \hline
60     & \textbf{70.172 (4)}   & 7200 (16)*          & \textbf{48.11 (4)}   & 7200 (25)*        & \textbf{62.687 (4)}   & 7200 (15)*   & \textbf{37.453 (2)}   & 7200 (20)*    & \textbf{166.766 (4)}   & 7200 (14)* \\ \hline
70     & \textbf{344.578 (6)}  & 7200 (17)*          & \textbf{45.953 (3)}  & 7200 (13)*        & \textbf{114.157 (4)}  & 7200 (14)*   & \textbf{129.265 (4)}  & 7200 (24)*    & \textbf{520.172 (5)}   & 7200 (13)* \\ \hline
80     & \textbf{575.875 (8)}  & 7200 (18)*          & \textbf{163.906 (4)} & 7200 (14)*        & \textbf{1421.5 (6)}   & 7200 (12)*   & \textbf{1463.5 (7)}   & 7200 (15)*    & \textbf{2379.765 (8)}  & 7200 (13)* \\ \hline
90     & \textbf{1799.312 (5)} & 7200 (12)*          & \textbf{179.922 (5)} & 7200 (9)*         & \textbf{3364.844 (6)} & 7200 (10)*   & 7200 (8)*             & 7200 (12)*    & \textbf{1892. 297 (7)} & 7200 (13)* \\ \hline
100    & \textbf{1543.235 (5)} & 7200 (11)*          & \textbf{962.062 (5)} & 7200 (5)*         & \textbf{720.406 (4)}  & 7200 (10)*   & \textbf{2831.515 (6)} & 7200 (11)*    & \textbf{2406.187 (7)}  & 7200 (13)* \\ \hline
110    & \textbf{1288.687 (5)} & 7200 (9)*           & \textbf{260.438 (4)} & 7200 (10)*        & \textbf{705.374 (4)}  & 7200 (10)*   & \textbf{3586.438 (7)} & 7200 (11)*    & 7200 (8)*              & 7200 (13)* \\ \hline
120    & \textbf{2373.344 (6)} & 7200 (10)*          & \textbf{527.531 (5)} & 2069.875 (10)     & \textbf{851 (4)}      & 7200 (9)*    & \textbf{718.875 (3)}  & 5169 (10)     & 7200 (8)*              & 7200 (13)* \\ \hline
130    & \textbf{161.641 (2)}  & 1394 .063(6)        & \textbf{67.265 (4)}  & 152.109 (11)      & \textbf{595.157 (4)}  & 7200 (5)*    & \textbf{4260 (8)}     & 5744.562 (11) & \textbf{644.422 (4)}   & 7200 (11)* \\ \hline
140    & 28.375 (2)            & \textbf{16.547 (2)} & 0.141 (1)            & \textbf{0.11 (1)} & \textbf{249.781 (3)}  & 3094.141 (9) & \textbf{714.016 (5)}  & 921.797 (6)   & \textbf{457.531 (3)}   & 7200 (11)* \\ \hline
\end{tabular}}
\end{center}
\caption{Comparison of run times (seconds) and number of plans visited of Algorithm \ref{alg:cncg} and the algorithm of \cite{cncg2} in instances with recruitment}
\label{tbl:compareTimeRec}
\end{table}

\begin{table}[]
\begin{center}
\resizebox{\textwidth}{!}{
\begin{tabular}{|c|c|c|c|c|c|c|c|c|c|c|}
\hline
Budget & Data1, Alg1          & Data1, \cite{cncg2}     & Data2, Alg1          & Data2, \cite{cncg2}             & Data3, Alg1            & Data3, \cite{cncg2}    & Data4, Alg1          & Data4, \cite{cncg2}              & Data5, Alg1          & Data5, \cite{cncg2}       \\ \hline\hline
50     & -           & 12.764\% & -           & 27.447\% & -           & 15.238\% & -           & 12.993\% & -           & 8.915\%  \\ \hline
60     & -           & 18.710\% & -           & 33.073\% & -           & 14.314\% & -           & 23.405\% & -           & 15.194\% \\ \hline
70     & -           & 21.093\% & -           & 33.642\% & -           & 18.612\% & -           & 17.951\% & -           & 20.278\% \\ \hline
80     & -           & 26.672\% & -           & 39.175\% & -           & 19.186\% & -           & 38.355\% & -           & 23.868\% \\ \hline
90     & -           & 32.142\% & -           & 56.361\% & -           & 20.347\% & 0.637\%     & 38.368\% & -           & 31.891\% \\ \hline
100    & -           & 36.823\% & -           & 69.444\% & -           & 28.804\% & -           & 35.693\% & -           & 37.486\% \\ \hline
110    & -           & 29.415\% & -           & 44.209\% & -           & 39.164\% & -           & 10.904\% & 0.929\%     & 46.313\% \\ \hline
120    & -           & 39.628\% & -           & -        & -           & 46.850\% & -           & -        & 7.865\%     & 52.938\% \\ \hline
130    & -           & -         & -           & -        & -           & 63.277\% & -           & -        & -           & 59.492\% \\ \hline
140    & -           & -        & -           & -        & -           & -        & -           & -        & -           & 65.755\% \\ \hline
\end{tabular}}
\end{center}
    \caption{Relative optimality gaps of unsolved instances of Algorithm \ref{alg:cncg} and the algorithm of \cite{cncg2} in instances with recruitment}
    \label{tbl:optGapRec}
\end{table}

\section{Computational Results: Application-based Analysis}
\label{analysis}
In our tests, we examine the limitations of applying the interdiction plan determined to be optimal by solving the traditional MFNIP. This will help determine the importance of accounting for restructuring in the context of disrupting drug trafficking networks. In each network, we observed that the optimal solution value of MFNIP-R is less than that of the optimal restructuring to MFNIP's optimal interdictions. In most instances, the increase from the MFNIP solution to the MFNIP-R solution is less than half of the increase from the MFNIP solution to the optimal restructuring after the MFNIP solution. We demonstrate the differences in these values with figures of one of the data sets, as the patterns we can observe from it are representative of all five data sets. Figure \ref{fig:200_1_base} compares flow values against the attacker budget in our base instance, which only allows for arcs to be restructured from a node in a given tier to a node in a tier immediately above or below it. Figure \ref{fig:200_1_wr} presents this comparison on instances with recruitment of new participants, i.e., new nodes can be added to the network, and Figure \ref{fig:200_1_or} presents this comparison on instances with organization restructuring, i.e., new nodes can be added and arcs can skip tiers. In each plot, we include the original maximum flow through the network (Max Flow), the maximum flow from the MFNIP solution (MFNIP Flow Before Restructuring), the maximum flow with the optimal restructuring after MFNIP interdictions (MFNIP Flow After Restructuring), and the maximum flow determined by MFNIP-R (MFNIP-R Flow).

\begin{figure}[h!]
\centering
\begin{subfigure}{.45\textwidth}
  \centering
  \includegraphics[width=\linewidth]{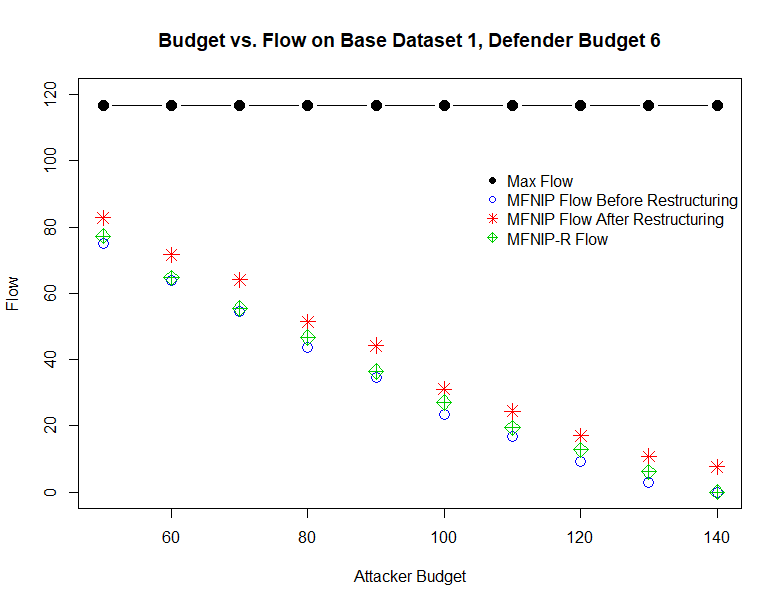}
  \caption{Base Instance}
  \label{fig:200_1_base}
\end{subfigure}%
\begin{subfigure}{.45\textwidth}
  \centering
  \includegraphics[width=\linewidth]{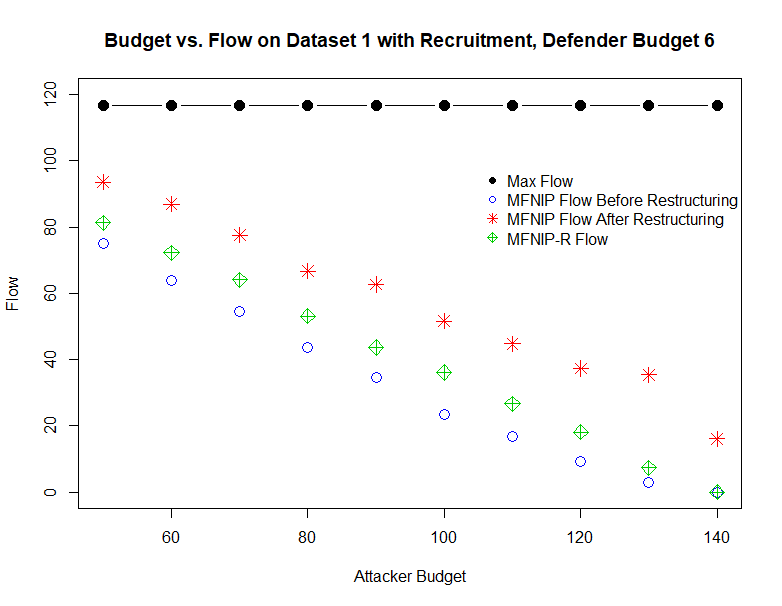}
  \caption{With Recruitment}
  \label{fig:200_1_wr}
\end{subfigure}

\begin{subfigure}{.45\textwidth}
  \centering
  \includegraphics[width=\linewidth]{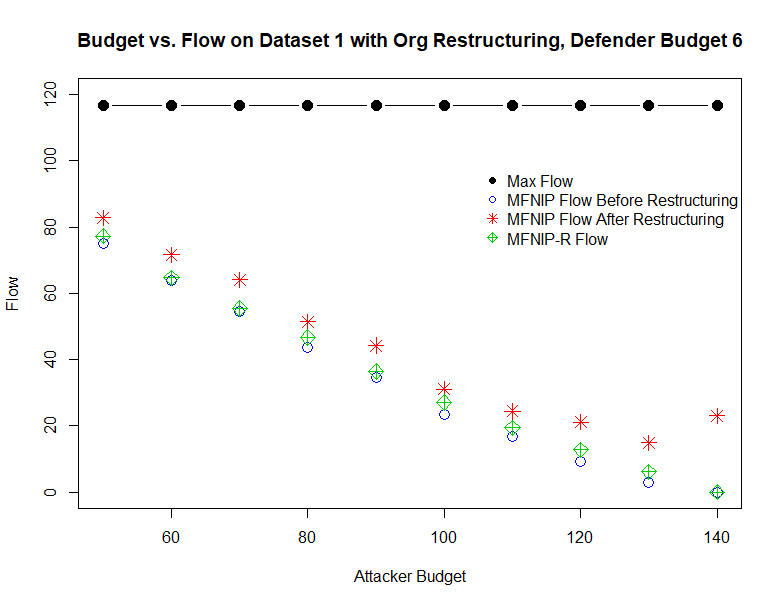}
  \caption{Organizational Restructuring}
  \label{fig:200_1_or}
\end{subfigure}%
\caption{Bilevel problem objective values over varying attacker budgets}
\label{fig:plots1}
\end{figure}

We can see that there is a significantly larger increase in flow after restructuring when recruitment of new participants is allowed. This supports the idea in criminology literature that interdicting at lower levels does not decrease the flow as much as is commonly assumed, since lower level participants are replaceable. Additionally, we note that the flow of MFNIP after restructuring with organization restructuring at an attacker budget of $140$ is larger than that flow with a budget of $130$. Because MFNIP does not account for restructuring, the interdiction plan identified using a budget of $140$ was able to initially disrupt more of the flow, but also allowed a different set of restructurings to occur. These restructurings resulted in significantly more flow in the resulting network. Even though we have a better capability to disrupt the network, this demonstrates that how we disrupt the network is just as important when we consider its ability to restructure after the interdiction. We also observed that the defender budget constraint was rarely fully used in the optimal bilevel solutions, indicating that we were also essentially solving the problem where we allow unlimited restructuring but constrain it based on where we have interdicted.

We note that there is very little difference in the flow values between the base data set and the adaptation with organizational restructuring. This is due to the fact that the interdiction decisions do not allow for these additional arcs to be allowed to be restructured by not disrupting high enough in the network. Ideally, we would want interdiction decisions to disrupt higher in the network, as we can see that disrupting lower in the network appears to reduce the flow significantly, but the presence of recruitable participants outside the network can drastically increase the flow after restructuring. To measure the impact of interdicting nodes that are in higher tiers, we include the constraint on interdicting leadership (see Appendix A for more details). Figure \ref{fig:200_1_base_il} plots the base instance flows with the interdicting leadership constraint, and Figure \ref{fig:200_1_or_il} plots the flows with organizational restructuring and the interdicting leadership constraint.

\begin{figure}[h!]
\centering
\begin{subfigure}{.45\textwidth}
  \centering
  \includegraphics[width=\linewidth]{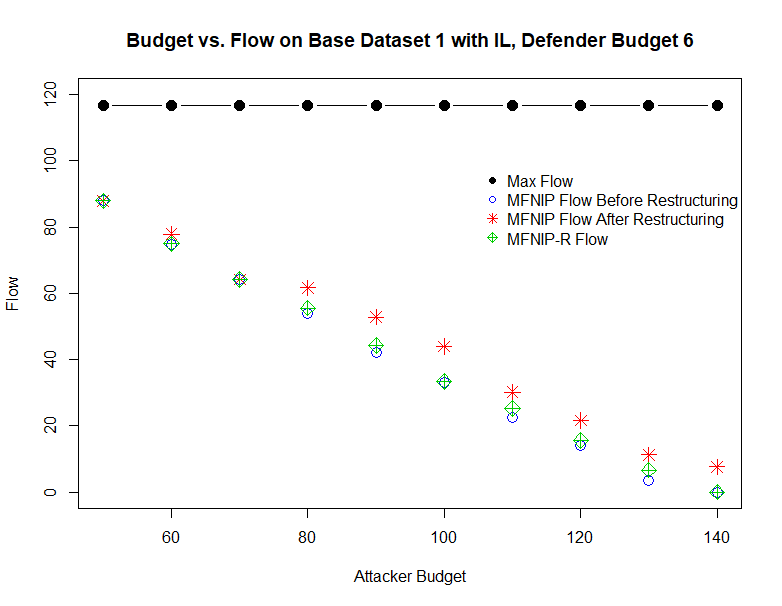}
  \caption{Base Instance}
  \label{fig:200_1_base_il}
\end{subfigure}%
\begin{subfigure}{.45\textwidth}
  \centering
  \includegraphics[width=\linewidth]{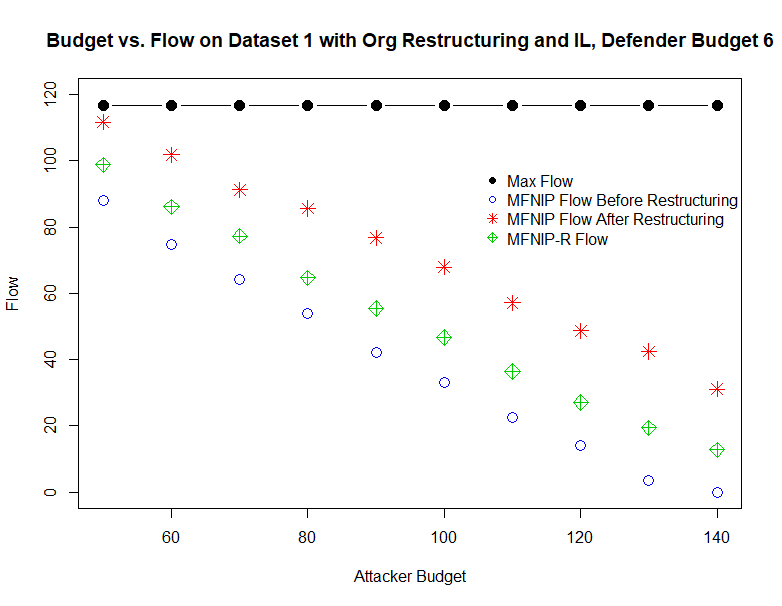}
  \caption{With Recruitment}
  \label{fig:200_1_or_il}
\end{subfigure}

\caption{Bilevel problem (with interdicting leadership) objective values over varying attacker budgets}
\label{fig:plots2}
\end{figure}

From here we can clearly see the impact of organizational restructuring. In the base instance, there was a relatively small gap between the MFNIP flow and the MFNIP flow after restructuring. However, when we allow for organization restructuring, there is a significantly larger gap, and the MFNIP-R solution is larger than in the base instances. However, we are still able to interdict in a way that mitigates the large gains in flow from restructuring after the interdictions proposed by MFNIP. This again demonstrates that how we interdict the network is much more important when we have to consider how the network will restructure after our interdictions.

In most MFNIP optimal solutions, the highest tier with interdictions is the safe house tier, the third tier from the bottom. This is due to the differences in capacities between the tiers; the flow can be decreased more drastically by interdicting nodes in the lower tiers than by using the equivalent amount of resources to interdict a node higher in the network. This is further exacerbated by the climbing the ladder constraints; it is much more costly to interdict higher in the network, further increasing the value of interdicting lower in the network. However, when recruiting outside participants is considered, we can see that more flow can be recovered than in the base instances. By interdicting more participants higher in the network, we expect that we will be able to better disrupt the network functionality, since these participants are more difficult to replace than lower level participants due to there being fewer participants in this tier. These gains from recruitment can be mitigated if efforts were made to limit lower level recruitment. By dedicating resources in ways that would decrease vulnerabilities of people to be recruited into the network, we can limit the ability of the network to replace these lower level participants, and focus on the disruption of higher level participants. Since the optimal flows in our base model are lower than the optimal flows in the model with recruitment, investing in communities to prevent recruitment into these networks would be more beneficial than reacting to recruitment of new users and dealers.

We must be careful when we interdict higher in the network when organizational restructuring is considered, as this type of restructuring allows for a significant amount of initial flow to be regained. By dedicating resources to determining which participants are most essential and least replaceable, we can exploit the vulnerabilities of the network structure to more effectively disruption the operations of the network.

\section{Conclusion and Future Work}
\label{conclude}
We define the max flow network interdiction problem with restructuring, where the defender is able to introduce new arcs to the network after interdiction decisions have been made. We demonstrate that, by not considering restructuring, interdictions can result in an \emph{increase} in flow, not a decrease as expected. Constraints on the decisions of the defender were defined based on modeling city-level drug trafficking networks. We also introduce the concepts of restructuring to recruit participants outside of the network and organizational restructuring, where participants are promoted and participants from other organizations can be recruited to a different organization. To solve this problem, we also improve upon the column-and-constraint generation procedure by utilizing partial information from infeasible restructuring plans.

For generated data on city-level drug trafficking networks consisting of $6$ tiers, we observe that the interdictions often disrupt up through the third lowest tier of the network, similar to the traditional MFNIP solution, but attain flow values less than that of the restructuring response to the optimal solution of MFNIP. This illustrates the importance of including restructuring in the network interdiction decision and makes interdiction models more aligned with reality. Without specifically requiring it, we are still mostly interdicting the ``replaceable" participants in the drug trafficking network. We suspect that, to inflict lasting damage on the operations of the network, we need to target higher up in the network. While we are able to assert more control over the flow; the optimal solution is still to not interdict the tiers of the network that are harder to replace. This is due to the fact that, for equivalent resources, flow decreases more when we interdict more nodes in lower tiers than fewer nodes in higher tiers, because of the additional costs from the climbing the ladder constraints. This leads to a future direction for research: can we define our mathematical problem to optimize some other measure while ensuring the flow is beneath some tolerable value to better disrupt the long-term operations of the network? 

We also observe that recruiting new participants and organizational restructuring in drug trafficking networks each result in a significant increase in flow after restructuring. While organizational restructuring may not occur to the extent we model, recruitment of new participants is a known issue. Our analysis suggests it is best to not only focus on disrupting the current network, but also prevent the network from recruiting participants. Future research into how to best disrupt the network and prevent the recruitment of new participants will be critical in developing comprehensive policies to addressing trafficking. Another direction of future research is learning the network. This algorithm, as well as many other solution algorithms to network interdiction problems, requires full knowledge of the network. Practically, when disrupting illicit trafficking networks, we do not have such knowledge at the beginning of our disruption efforts. Studying how we can make interdiction decisions while learning the network will be vital to the application in disrupting illicit trafficking networks.

In addition to drug trafficking, we anticipate that network interdiction will be useful in disrupting other illicit trafficking networks, such as human trafficking. More conceptual and empirical evidence is needed to apply MFNIP-R to the issue of human trafficking, particularly in the case of human trafficking for sexual exploitation (or sex trafficking). Compared to drug trafficking, it is less clear in sex trafficking networks what an interdiction might seek to restrict in terms of flow. The function of a drug trafficking network is to move and eventually sell a product (i.e. drugs) to a user. Thus the flow is the drugs. In a sex trafficking network, flow is more complex. The ``product" being sold in sex trafficking is people and a sexual experience \cite{martin2017mapping, martin2014mapping}. The commercialization and sale of sex to a sex buyer may in some ways be analogous to the sale of drugs since both are commodities. The flow within the network is not the same because people (and their labor) are not equivalent to drugs. Much more work is needed to conceptualize how this concept of ``flow" should best be applied to sex trafficking networks. Further, given the differences between drug and human trafficking, ethical considerations are also warranted as more mathematical modeling is applied to the problem domain of human trafficking.

\section*{Acknowledgements}
This material is based upon work supported by the National Science Foundation (NSF) under Grant No. 1838315 and the National Institute of Justice (NIJ) under Grant No. 2020-R2-CX-0022. The opinions expressed in the paper do not necessarily reflect the views of the NSF or NIJ. We recognize that our research cannot capture all the complexities of the lived experiences of trafficking victims and survivors. The model developed here was originally designed to explore general illicit trafficking networks, including both drug trafficking and sex trafficking networks. However, given the lack of data on sex trafficking networks and ethical considerations needed to apply OR models to sex trafficking networks, we tested the model with a comprehensive data set on drug trafficking. We would like to acknowledge Christina Melander and Emily Singerhouse for their help in understanding the complexities of sex trafficking and helping us identify the additional work and ethical considerations required in applying OR models to disrupting sex trafficking networks. 

\bibliographystyle{acm} 
\bibliography{ref.bib} 

\newpage
\begin{appendices}
\section{Proof of Lemma}
\label{apdx:pf}

Proof of Lemma \ref{lm:noInc}:
\begin{proof}
We first show that restructuring an arc $(i,j)$ that satisfies one of the three conditions in Lemma \ref{lm:noInc} will not increase the flow. Consider an arc $(i,j)$, where $i \in \bar{U}$ (Condition 1). Since $i \in \bar{U}$, $i$ is not reachable from $s$ in $D_f$, and thus there is no path from $s$ to $i$ in $D_f$. Note that including $(i, j)$ in our network introduces the same arc $(i, j)$ in the auxiliary network $D_f$. Since there is no path from $s$ to $i$ in $D_f$, including $(i, j)$ will not complete a path from $s$ to $t$ in $D_f$. Thus, restructuring arc $(i,j)$ will not increase the value of the flow.

Next, consider an arc $(i,j)$, where $i \in U$ and $j \in U$ (Condition 2). Since $i \in U$, $i$ is reachable from $s$ in $D_f$, and thus there exists a path from $s$ to $i$ in $D_f$. Likewise, since $j \in U$, then $j$ is reachable from $s$ in $D_f$, and thus there exists a path from $s$ to $j$ in $D_f$. This implies that there cannot exist a path from $j$ to $t$ in $D_f$. If there were such a path, then we would have a path from $s$ to $t$, and thus the current flow is not optimal. Since there is no path from $j$ to $t$ in $D_f$, including arc $(i, j)$ will not complete a path from $s$ to $t$ in $D_f$. Thus, the value of the flow will not increase if this arc is introduced.

Lastly, consider an arc $(i,j)$, where $i \in U$ and $j \in \bar{U}$, and there is no path from $j$ to $t$ in $D_f$ (Condition 3). Since $i \in U$, there exists a path from $s$ to $i$ in $D_f$, and thus there is no path from $i$ to $t$ in $D_f$, else the original flow is not optimal. Since there is no path from $j$ to $t$ in $D_f$, including arc $(i, j)$ will not complete a path from $s$ to $t$ in $D_f$. Thus, the value of the flow will not increase. Thus, restructuring an arc $(i,j)$ that satisfies one of the three conditions will not increase the flow.

We now show that if restructuring an arc $(i,j)$ will not increase the flow, then it must satisfy one of the three conditions. Suppose that restructuring arc $(i,j)$ will not increase the flow. This implies that adding $(i,j)$ to $D_f$ does not complete an $s$-$t$ path. Thus, there is either no path from $s$ to $i$ in $D_f$ or no path from $j$ to $t$ in $D_f$. First, suppose there is no path from $s$ to $i$ in $D_f$. Since no such path exists, $i \in \bar{U}$, thus satisfying Condition 1. Next, suppose there is a path from $s$ to $i$ in $D_f$, meaning that $i \in U$. Then there must not be a path from $j$ to $t$ in $D_f$. We have two cases: $j \in U$ or $j \in \bar{U}$. If $j \in U$, then Condition 2 is satisfied. If $j \in \bar{U}$, since there must not be a path from $j$ to $t$ in $D_f$, Condition 3 is satisfied. Thus, if restructuring an arc $(i,j)$ will not increase the flow, then it must satisfy Conditions $1$, $2$ or $3$.
\end{proof}

\newpage
\section{Computational Results}
\label{fullComp}
Entries marked with * are instances where the method did not solve within the two hour time limit. We instead list the upper bound, as this is the best known bilevel feasible solution, and the time of last completed iteration.
\begin{singlespace}
\begin{table}[H]
 \begin{center}
  \begin{tabular}{|c|c|c|c|c|c|} 
  \hline
  Budget & 1 & 2 & 3 & 4 & 5\\ [0.5ex] 
  \hline\hline
  Base & 116.77137 & 111.89403 & 114.435134 & 110.546185 & 118.280546\\
  \hline
  50 & 77.2257 & 49.64045 & 71.548639 & 63.32806 & 81.92476 \\ 
  \hline
  60 & 64.72307 & 40.59544 & 63.378588 & 53.63989 & 71.12085 \\
  \hline
  70 & 55.44162 & 32.54936 & 54.32644 & 44.69965 & 62.40129 \\
  \hline
  80 & 46.65828 & 21.85086 & 46.196279 & 36.37513 & 53.86366 \\
  \hline
  90 & 36.38575 & 13.76413 & 38.230575 & 27.284 & 43.39701 \\
  \hline
  100 & 27.08376 & 6.3368 & 29.14647 & 18.14417 & 36.78273 \\
  \hline
  110 & 19.46085 & 0 & 21.07597 & 11.48059 & 29.61927 \\
  \hline
  120 & 12.79303 & 0 & 14.95574 & 7.04445 & 21.7783 \\
  \hline
  130 & 6.14819 & 0 & 9.099645 & 1.90682 & 14.82776 \\
  \hline
  140 & 0 & 0 & 2.844335 & 0 & 9.484119 \\
  \hline
 \end{tabular}
 \caption{Optimal flows for different base data sets with 200 users with fixed defender budget 6}
 \end{center}
\end{table}

\begin{table}[H]
 \begin{center}
  \begin{tabular}{|c|c|c|c|c|c|} 
  \hline
  Budget & 1 & 2 & 3 & 4 & 5\\ [0.5ex] 
  \hline\hline
  50 & 110.875 (3) & 10.938 (3) & 25.35 (3) & 18.86 (3) & 139.688 (3) \\ 
  \hline
  60 & 18.844 (3) & 5.672 (2) & 521.078 (8) & 23.75 (3) & 70.422 (2) \\
  \hline
  70 & 12.781 (2) & 22.968 (3) & 449.516 (7) & 51.765 (4) & 177.031 (3) \\
  \hline
  80 & 48.313 (3) & 47.078 (3) & 1316.953 (7) & 104.438 (4) & 170.422 (3) \\
  \hline
  90 & 45.203 (2) & 122.359 (3) & 688.641 (5) & 139.515 (4) & 56.063 (2) \\
  \hline
  100 & 61.188 (2) & 95.031 (3) & 69.656 (2) & 722.641 (5) & 210.328 (3) \\
  \hline
  110 & 80.078 (2) & 3.219 (2) & 118.578 (2) & 409.906 (4) & 312.594 (3) \\
  \hline
  120 & 378.562 (3) & 0.156 (1) & 90.485 (2) & 968.485 (6) & 97.468 (2) \\
  \hline
  130 & 636.907 (3) & 0.219 (1) & 1893.594 (7) & 537.984 (6) & 219.719 (3)\\
  \hline
  140 & 14.343 (2) & 0.141 (1) & 4832.516 (12) & 62.437 (4) & 758.61 (5) \\
  \hline
 \end{tabular}
 \caption{Run time in seconds (number of plans visited) for different base data sets with 200 users with fixed defender budget 6}
 \end{center}
\end{table}

\newpage
\begin{table}[H]
 \begin{center}
  \begin{tabular}{|c|c|c|c|c|c|} 
  \hline
  Budget & 1 & 2 & 3 & 4 & 5\\ [0.5ex] 
  \hline\hline
   Base & 116.77137 & 111.89403 & 114.435134 & 110.546185 & 118.280546\\
  \hline
  50 & 81.30466 & 63.862729 & 77.551249 & 69.72654 & 83.62605 \\ 
  \hline
  60 & 72.282593 & 53.256709 & 67.19034 & 61.85482 & 75.2244 \\
  \hline
  70 & 64.070003 & 44.619742 & 59.29457 & 54.50484 & 67.24456 \\
  \hline
  80 & 53.155684 & 36.102209 & 53.44689 & 47.1921 & 59.20083 \\
  \hline
  90 & 43.709545 & 27.680039 & 45.55112 & 38.78783* & 51.62168 \\
  \hline
  100 & 36.09047 & 21.568024 & 36.29383 & 30.01219 & 44.75593 \\
  \hline
  110 & 26.64771 & 12.66219 & 28.37822 & 22.45318 & 39.54587* \\
  \hline
  120 & 18.16082 & 4.58978 & 20.48245 & 14.29821 & 36.16806* \\
  \hline
  130 & 7.3016 & 0 & 15.68868 & 9.682533 & 26.62906 \\
  \hline
  140 & 0 & 0 & 6.92709 & 3.27475 & 16.56582 \\
  \hline
 \end{tabular}
 \caption{Optimal flows for different data sets with recruitment with 200 users with fixed defender budget 6}
 \end{center}
\end{table}

\begin{table}[H]
 \begin{center}
  \begin{tabular}{|c|c|c|c|c|c|} 
  \hline
  Budget & 1 & 2 & 3 & 4 & 5\\ [0.5ex] 
  \hline\hline
  50 & 55.5 (3) & 37.75 (4) & 211.891 (6) & 70.407 (4) & 161.484 (5) \\ 
  \hline
  60 & 70.172 (4) & 48.11 (4) & 62.687 (4) & 37.453 (2) & 166.766 (4) \\
  \hline
  70 & 344.578 (6) & 45.953 (3) & 114.157 (4) & 129.265 (4) & 520.172 (5) \\
  \hline
  80 & 575.875 (8) & 163.906 (4) & 1421.5 (6) & 1463.5 (7) & 2379.765 (8) \\
  \hline
  90 & 1799.312 (5) & 179.922 (5) & 3364.844 (6) & 7200 (8)* & 1892.297 (7) \\
  \hline
  100 & 1543.235 (5) & 962.062 (5) & 720.406 (4) & 2831.515 (6) & 2406.187 (7) \\
  \hline
  110 & 1288.687 (5) & 260.438 (4) & 705.375 (4) & 3586.438 (7) & 7200 (8)* \\
  \hline
  120 & 2373.344 (6) & 527.531 (5) & 851 (4) & 718.875 (3) & 7200 (9)* \\
  \hline
  130 & 161.642 (2) & 67.265 (4) & 595.157 (4) & 4260 (8) & 644.422 (4) \\
  \hline
  140 & 28.375 (2) & 0.141 (1) & 249.781 (3) & 714.016 (5) & 457.531 (3) \\
  \hline
 \end{tabular}
 \caption{Run time in seconds (number of plans visited) for different data sets with recruitment with 200 users with fixed defender budget 6}
 \end{center}
\end{table}

\newpage
\begin{table}[H]
 \begin{center}
  \begin{tabular}{|c|c|c|c|c|c|} 
  \hline
  Budget & 1 & 2 & 3 & 4 & 5\\ [0.5ex] 
  \hline\hline
  Base & 116.77137 & 111.89403 & 114.435134 & 110.546185 & 118.280546\\
  \hline
  50 & 77.2257 & 49.64045 & 71.548639 & 63.32806 & 81.92476 \\ 
  \hline
  60 & 64.72307 & 40.59544 & 63.378588 & 53.63989 & 71.12085 \\
  \hline
  70 & 55.44162 & 32.54936 & 54.32644 & 44.69965 & 62.40129 \\
  \hline
  80 & 46.65728 & 21.85086 & 46.196279 & 36.37513 & 53.86366 \\
  \hline
  90 & 36.38575 & 13.76413 & 38.230575 & 27.284 & 43.397010 \\
  \hline
  100 & 27.08376 & 6.33368 & 30.637375 & 18.14417 & 36.78273 \\
  \hline
  110 & 19.46085 & 0 & 23.122495 & 11.48059 & 29.61927 \\
  \hline
  120 & 12.79303 & 0 & 15.856735 & 7.04445 & 21.7783 \\
  \hline
  130 & 6.14819 & 0 & 9.099645 & 1.90682 & 14.82776 \\
  \hline
  140 & 0 & 0 & 2.844335 & 0 & 9.484119 \\
  \hline
 \end{tabular}
 \caption{Optimal flows for different data sets with organizational restructuring with 200 users with fixed defender budget 6}
 \end{center}
\end{table}

\begin{table}[H]
 \begin{center}
  \begin{tabular}{|c|c|c|c|c|c|} 
  \hline
  Budget & 1 & 2 & 3 & 4 & 5\\ [0.5ex] 
  \hline\hline
  50 & 93.812 (3) & 10.797 (3) & 30.39 (3) & 14.938 (3) & 91.469 (3) \\ 
  \hline
  60 & 18.438 (3) & 6.453 (2) & 735.063 (8) & 18.828 (3) & 55.578 (2) \\
  \hline
  70 & 11.328 (2) & 27.547 (3) & 448.094 (7) & 50.515 (4) & 178.766 (3) \\
  \hline
  80 & 46.109 (3) & 49.735 (3) & 1185.343 (7) & 116.766 (4) & 168.125 (3) \\
  \hline
  90 & 48.031 (2) & 119.812 (3) & 644.937 (5) & 162.078 (4) & 56.422 (2) \\
  \hline
  100 & 49.984 (2) & 114.078 (3) & 1669.125 (7) & 706.016 (5) & 234.938 (3) \\
  \hline
  110 & 80.125 (2) & 4.688 (2) & 1906.25 (7) & 407.812 (4) & 338.156 (3) \\
  \hline
  120 & 424.891 (3) & 0.125 (1) & 1782.188 (7) & 981.297 (6) & 131.39 (2) \\
  \hline
  130 & 150.859 (2) & 2.063 (3) & 3860.046 (7) & 437.812 (6) & 564.625 (3) \\
  \hline
  140 & 197.172 (3) & 0.125 (1) & 3247.78 (10) & 17.328 (4) & 892.281 (4) \\
  \hline
 \end{tabular}
 \caption{Run time in seconds (number of plans visited) for different data sets with organizational restructuring with 200 users with fixed defender budget 6}
 \end{center}
\end{table}

\newpage
\begin{table}[H]
 \begin{center}
  \begin{tabular}{|c|c|c|c|c|c|} 
  \hline
  Budget & 1 & 2 & 3 & 4 & 5\\ [0.5ex] 
  \hline\hline
 Base & 116.77137 & 111.89403 & 114.435134 & 110.546185 & 118.280546\\
  \hline
  50 & 87.99789 & 69.30883 & 85.118192 & 71.85268 & - \\ 
  \hline
  60 & 74.99896 & 60.12165 & 66.37184 & 60.58043 &  101.51818\\
  \hline
  70 & 64.18631 & 49.64045 & 54.87111 & 51.50233 & 85.87106 \\
  \hline
  80 & 55.52443 & 32.50457 & 47.2207 & 43.40857 & 75.35971 \\
  \hline
  90 & 44.27122 & 18.67813 & 38.47231 & 33.37398 & 62.93186 \\
  \hline
  100 & 33.4231 & 10.19944 & 29.14647 & 25.4703 & 51.48153 \\
  \hline
  110 & 25.27449 & 3.26572 & 21.07597 & 17.65032 & 42.37629 \\
  \hline
  120 & 15.62627 & 0 & 14.95574 & 10.339423 & 33.52178 \\
  \hline
  130 & 6.5985 & 0 & 9.2884 & 5.91753 & 23.42954 \\
  \hline
  140 & 0 & 0 & 3.08835 & 1.90682 & 16.44772 \\
  \hline
 \end{tabular}
 \caption{Optimal flows for different base data sets with interdicting leadership with 200 users with fixed defender budget 6}
 \end{center}
\end{table}

\begin{table}[H]
 \begin{center}
  \begin{tabular}{|c|c|c|c|c|c|} 
  \hline
  Budget & 1 & 2 & 3 & 4 & 5\\ [0.5ex] 
  \hline\hline
  50 & 0.671 (1) & 0.656 (1) & 4.188 (2) & 6.672 (3) & - \\ 
  \hline
  60 & 4.766 (2) & 1.437 (1) & 5.468 (2) & 10.359 (3) & 3.813 (2) \\
  \hline
  70 & 1.84 (1) & 36.157 (3) & 4.438 (2) & 29.563 (4) & 4.234 (2) \\
  \hline
  80 & 17.64 (2) & 2.031 (1) & 17.062 (2) & 71.203 (3) & 36.266 (4) \\
  \hline
  90 & 54.625 (2) & 2.218 (1) & 176.64 (3) & 60.469 (3) & 25.796 (2) \\
  \hline
  100 & 59.375 (2) & 39.797 (2) & 34.672 (2) & 706.828 (5) & 102.11 (3) \\
  \hline
  110 & 79.641 (2) & 21.5 (2) & 167.219 (2) & 555.609 (5) & 47.375 (2) \\
  \hline
  120 & 150.265 (2) & 3.031 (2) & 119.984 (2) & 273.141 (4) & 123.75 (2) \\
  \hline
  130 & 163.765 (2) & 0.203 (1) & 153.328 (2) & 350.984 (4) & 206.25 (2) \\
  \hline
  140 & 11.094 (3) & 0.141 (1) & 148.078 (3) & 722.672 (7) & 111.968 (2) \\
  \hline
 \end{tabular}
 \caption{Run time in seconds (number of plans visited) for different base data sets with interdicting leadership with 200 users with fixed defender budget 6}
 \end{center}
\end{table}

\newpage
\begin{table}[H]
 \begin{center}
  \begin{tabular}{|c|c|c|c|c|c|} 
  \hline
  Budget & 1 & 2 & 3 & 4 & 5\\ [0.5ex] 
  \hline\hline
  Base & 116.77137 & 111.89403 & 114.435134 & 110.546185 & 118.280546\\
  \hline
  50 & 98.795814 & 69.30883 & 85.118192 & 84.03023 & - \\ 
  \hline
  60 & 86.14281 & 60.12165 & 74.435869 & 71.85268 & 101.51818 \\
  \hline
  70 & 77.2257 & 49.64045 & 65.118192 & 60.58043 & 91.52576 \\
  \hline
  80 & 64.723070 & 40.59544 & 54.435869 & 51.50233 & 78.08339 \\
  \hline
  90 & 55.441619 & 32.54936 & 44.666794 & 38.18203 & 67.50571 \\
  \hline
  100 & 46.65828 & 21.85086 & 36,691006 & 29.37457 & 57.97674 \\
  \hline
  110 & 36.38575 & 13.76413 & 29.14647 & 20.19363 & 47.28458 \\
  \hline
  120 & 27.08376 & 6.33368 & 21.07597 & 14.29806 & 38.26267 \\
  \hline
  130 & 19.46085 & 0 & 14.95574 & 7.04445 & 30.4055 \\
  \hline
  140 & 12.79303 & 0 & 8.972234 & 3.27475 & 22.50035 \\
  \hline
 \end{tabular}
 \caption{Optimal flows for different data sets with organizational restructuring and interdicting leadership with 200 users with fixed defender budget 6}
 \end{center}
\end{table}

\begin{table}[H]
 \begin{center}
  \begin{tabular}{|c|c|c|c|c|c|} 
  \hline
  Budget & 1 & 2 & 3 & 4 & 5\\ [0.5ex] 
  \hline\hline
  50 & 74.656 (5) & 1.156 (1) & 4.406 (2) & 26.25 (5) & - \\ 
  \hline
  60 & 17.782 (3) & 1.984 (1) & 50.344 (6) & 67.969 (5) & 4.812 (2) \\
  \hline
  70 & 375.093 (6) & 44.64045 (4) & 68.75 (6) & 61.969 (5) & 606.36 (6) \\
  \hline
  80 & 139.875 (5) & 29.015 (2) & 21.406 (3) & 154.484 (6) & 285.156 (6) \\
  \hline
  90 & 102.719 (3) & 783.141 (6) & 52.782 (3) & 244.438 (7) & 412.516 (6) \\
  \hline
  100 & 1357.375 (6) & 573.609 (5) & 559.437 (5) & 514.016 (5) & 1192.407 (6) \\
  \hline
  110 & 840.975 (5) & 824.36 (6) & 1401.797 (7) & 507 (5) & 885.39 (5) \\
  \hline
  120 & 1600.578 (6) & 926.594 (6) & 625.078 (5) & 2963.718 (8) & 570.11 (4) \\
  \hline
  130 & 1661.359 (6) & 9.328 (3) & 1954.64 (7) & 925.922 (5) & 3720.812 (7) \\
  \hline
  140 & 2580.844 (7) & 0.125 (1) & 2745.125 (9) & 1535.969 (9) & 2004.015 (5) \\
  \hline
 \end{tabular}
 \caption{Run time in seconds (number of plans visited) for different data sets with organizational restructuring and interdicting leadership with 200 users with fixed defender budget 6}
 \end{center}
\end{table}
\end{singlespace}

\end{appendices}
\end{document}